\documentclass{amsart}
\usepackage{amsmath,amsthm}
\usepackage{amsfonts,amssymb}
\usepackage{mathrsfs}
\usepackage{amsbsy}

\usepackage{accents}
\usepackage{enumerate}
\usepackage{accents,color}
\usepackage{graphicx}

\newcommand{\lvt}{\left|\kern-1.35pt\left|\kern-1.3pt\left|}
\newcommand{\rvt}{\right|\kern-1.3pt\right|\kern-1.35pt\right|}

\newtheorem{thm}{Theorem}[section]
\newtheorem{cor}[thm]{Corollary}
\newtheorem{lem}[thm]{Lemma}
\newtheorem{prop}[thm]{Proposition}

\newtheorem{defn}[thm]{Definition}

\theoremstyle{remark}

 \def\la{{\langle}}
 \def\ra{{\rangle}}
 \def\ve{{\varepsilon}}

 \def\d{\mathrm{d}}
 
 \def\sph{{\mathbb{S}^{d-1}}}

 \def\sP{{\mathsf P}}
 \def\sQ{{\mathsf Q}}
 
 \def\sT{{\mathsf T}}
 
 \def\bs{{\mathsf b}}
 \def\sc{{\mathsf c}}
 \def\sd{{\mathsf d}}
 \def\sm{{\mathsf m}}
 \def\sw{{\mathsf w}}

 \def\a{{\alpha}}
 \def\b{{\beta}}
 \def\g{{\gamma}}
 
 \def\t{{\theta}}
 \def\l{{\lambda}}
 
 \def\s{\sigma}
 \def\la{{\langle}}
 \def\ra{{\rangle}}
 \def\ve{{\varepsilon}}
 
 \def\bb{{\mathbf b}}
 \def\cb{{\mathbf c}}

 \def\Mb{{\mathbf M}}
 \def\Pb{{\mathbf P}}
 \def\Qb{{\mathbf Q}}
 
 \def\Tb{{\mathbf T}}

 \def\CV{{\mathcal V}}

 \def\BB{{\mathbb B}}

 \def\NN{{\mathbb N}}

 \def\RR{{\mathbb R}}

 \def\VV{{\mathbb V}}

  \def\CM{{\mathscr M}}
  
\DeclareMathAlphabet{\mathscrbf}{OMS}{mdugm}{b}{n}
\def\bCM{{\mathscrbf}}
 
\def\proj{\operatorname{proj}}

\def\lla{\langle{\kern-2.5pt}\langle}      
\def\rra{\rangle{\kern-2.5pt}\rangle}

\DeclareMathAlphabet{\mathscrbf}{OMS}{mdugm}{b}{n}
\def\bCM{{\mathscrbf{M}}}

\def\f{\frac}

\graphicspath{{./}}
\begin{document}
 
\title[Maximal functions and multiplier theorem on conic domains]
{Maximal functions and multiplier theorem for Fourier orthogonal series}

\author{Yuan~Xu}
\address{Department of Mathematics, University of Oregon, Eugene, 
OR 97403--1222, USA}
\email{yuan@uoregon.edu} 
\thanks{The author is partially supported by Simons Foundation Grant \#849676.}

\date{\today}  

\begin{abstract} 
Under the assumption that orthogonal polynomials of several variables admit an addition formula, we can define 
a convolution structure and use it to study the Fourier orthogonal expansions on a homogeneous space. We 
define a maximal function via the convolution structure induced by the addition formula and use it to establish 
a Marcinkiewicz multiplier theorem. For the homogeneous space defined by a family of weight functions 
on conic domains, we show that the maximal function is bounded by the Hardy-Littlewood maximal function so 
that the multiplier theorem holds on conic domains.
\end{abstract}
 
\maketitle
  
\section{Introduction}
\setcounter{equation}{0}

Fourier orthogonal expansions on a domain in $\RR^d$ have been studied in depth for several regular 
domains in high dimensions in recent years. Let $\Omega$ be either a hypersurface or a regular domain 
with open interior in $\RR^d$ and let $\sw$ be a weight function on $\Omega$. We consider the orthogonal 
structure with respect to the inner product 
\begin{equation}\label{eq:ipd-w}
        \la f, g\ra_\sw = \int_{\Omega} f(x) g(x) \sw(w) \d x,
\end{equation}
well-defined for the space $\Pi^d(\Omega)$ of polynomials polynomials restricted on $\Omega$. 
Let $\CV_n^d$ be the subspace of $\Pi^d(\Omega)$ that consists of orthogonal polynomials of degree 
$n$ with respect to the inner product. With respect to an orthogonal basis $\{P_k^n: 1 \le k \le \dim \CV_n^d\}$
of $\CV_n^d$, the Fourier orthogonal expansion of $f\in L^2(\Omega, \sw)$ is defined by
\begin{equation}\label{eq:FourierOP}
  f = \sum_{n=0}^\infty \proj_n(\sw; f), \qquad \proj_n (\sw;f) =
   \sum_{k=0}^n \sum_{j=1}^{\dim \CV_n^d} \f{\la f, P_k^n\ra_\sw}{\la P_k^n, P_k^n\ra_\sw} P_k^n. 
\end{equation}
The orthogonal projection operator $\proj_n: L^2(\Omega,\sw) \mapsto \CV_n^d$ can be written as
\begin{equation}\label{eq:proj}
  \proj_n(\sw; f, x) = \int_\Omega f(y) P_n(\sw; x ,y) \sw(y) \d y, 
\end{equation}
where $P_n(\sw; \cdot,\cdot)$ is the reproducing kernel of $\CV_n^d$ and it is given by 
\begin{equation}\label{eq:Pn-kernel}
P_n(\sw; x,y):=\sum_{j=1}^{\dim \CV_n^d} \f{ P_k^n(x)P_k^n(y)}{\la P_k^n, P_k^n\ra_\sw}. 
\end{equation}
This kernel is uniquely determined and is independent of the choice of orthogonal bases. To develop an 
analysis of orthogonal expansions beyond the standard $L^2$ theory, we need further information on the 
orthogonal structure that can hold only for special $\sw$ and $\Omega$. The first domain on which a 
systematic study is carried out is the spherical harmonic expansion on the unit sphere $\sph$; see, for example, 
\cite{BBP, BC, DaiX, Rus, Stein} and references in \cite{DaiX, Stein}. The approach developed on the sphere
is extended to study the Fourier-orthogonal expansions on several other regular domains, including the unit 
sphere with product type weight functions, the unit ball and the regular simplex with classical weight functions; 
see, for example, \cite{DaiX, DX} and references there as well as \cite{KPX, KPX2, NSS, PX2}. More recently, 
they are studied for Jacobi type weight functions on conic domains in \cite{X20a, X20b, X21}. 
  
A common trait that makes the study on these regular domains possible is the {\it addition formula} for the 
orthogonal polynomials, which gives a closed-form formula for the reproducing kernels. The simplest case 
is the spherical harmonic expansions on the unit sphere $\sph$, where the reproducing kernel satisfies
\begin{equation}\label{eq:additionS}
  P_n(x,y) = \frac{n+\l}{\l}C_n^\l(x,y), \qquad \l = \f{d-2}{2}, \quad x,y \in \sph,
\end{equation}
in which $C_n^\l$ is the Gegenbauer polynomial of degree $n$. This identity is the classical addition formula 
for spherical harmonics, since it gives a closed-form formula for the sum in \eqref{eq:Pn-kernel}. It is also 
the reason why we retain the name addition formula for closed-formula for other orthogonal polynomials. The 
remarkable formula \eqref{eq:additionS} allows us to relate much of the 
study on the unit sphere to the Fourier-Gegenbauer series in one variable (cf. \cite{BC, DaiX, Stein} 
and references there). Analysis on other regular domains also relies on their addition formulas, which are 
however more complex and given by integrals of the Gegenbauer or the Jacobi polynomials.  Recently, 
motivated by newly discovered addition formulas for the conic domains in \cite{X20a, X20b}, we have shown
in \cite{X21} that a substantial portion of analysis holds in a general setting, under a mere assumption that an 
addition formula for the reproducing kernel exists. The main result of the present paper is to show that this 
paradigm applies to the maximal functions and multiplier operators. 

We will assume that the orthogonal structure of $L^2(\Omega,\sw)$ admits an addition formula given 
as an integral of a Jacobi polynomial (see Definition \ref{defn:additionF} below). It leads to a pseudo 
convolution structure and a maximal function that is suitable for studying the Fourier orthogonal expansions. 
Under a technical assumption, we can then prove a Marcinkiewicz multiplier theorem, following the original 
approach in \cite{BC} for the unit sphere closely. We then show that the general setting applies to the conic surface 
$$
   \VV_0^{d+1} = \left \{ (x,t) \in \RR^{d+1}: \|x\| \le t, \, x \in \BB^d, \, 0 \le t \le 1 \right\}
$$
in $\RR^{d+1}$ with the weight function $t^{-1} (1-t)^\g$, $\g \ge 0$, and the cone bounded by $\VV_0^{d+1}$ 
and the hyperplane $t=1$ with the weight function $(t^2-\|x\|^2)^{\mu-\f12}(1-t)^\g$, $\mu, \g \ge 0$. For these 
conic domains, we shall show that the maximal function defined via the pseudo convolution is bounded 
by the Hardy-Littlewood maximal function. The latter is sufficient for proving the multiplier theorem on the 
conic domains. 

The paper is organized as follows. The theory for the general setting is developed in the next section. 
The conic surface is treated in Section 3 and the cone in Section 4. Throughout this paper, we shall use
$c, c_1, c_2$ to denote positive constants that depend only on fixed parameters, and their values may 
change from line to line. Furthermore, we write $A\sim B$ if $c_1 A \le B\le c_2 A$.

\section{Homogenous space admitting an addition formula} 

In this section, we consider analysis on a homogeneous space that admits an addition formula. The
first subsection is preliminary, where we give basic definitions that will be used throughout the rest
of the paper. The addition formula and the convolution structure are defined and discussed in the 
second subsection. The Poisson integrals and the Ces\`aro means are discussed in subsections three
and four, respectively. The latter contains an assertion that will be needed for the proof of the multiplier 
theorem in subsection five. Finally, a maximal function is defined and discussed in subsection six. 

\subsection{Preliminary} 
Let $\Omega$ be a domain in $\RR^d$ with a metric $\sd$ and let $\sw$ be a nonnegative weight function 
on $\Omega$. The space $(\Omega,\sw,\sd)$ is called a homogeneous space if all open balls 
$B(x,r) = \{y\in \Omega: \sd(x,y) < r\}$  are measurable and $\sw$ is a doubling weight. The latter means 
that $\sw(B(x,2r)) \le c\, \sw (B(x,r))$ for all $x \in \Omega$ and $r > 0$, where $c$ is dependent of $x$ and $r$,
and
$$
       \sw (E) := \int_E \sw(x) \d \sm(x), \qquad E \subset \Omega,
$$
where $\sm$ is the Lebesgue measure on $\Omega$. For our purpose, the domain $\Omega$ is either a 
quadratic surface, such as the unit sphere or a conic surface, or a domain bounded by such a surface and, 
when needed, hyperplanes. Moreover, our weight function will also be specified. 

Let $\Pi(\Omega)$ denote the space of polynomials restricted to $\Omega$ and let $\Pi_n(\Omega)$ denote 
its subspace of polynomials of degree at most $n$. If the interior of $\Omega$ is open, then $\Pi(\Omega)$ 
contains all polynomials of degree $n$ in $d$ variables and we write $\Pi_n = \Pi_n(\Omega)$. If $\Omega$ 
is an algebraic surface, such as the sphere $\sph$, then $\Pi_n(\Omega)$ contains all polynomials restricted 
to the surface. 

Let $\sw$ be a doubling weight on $\Omega$ normalized so that $\int_\Omega \sw(x) \d \sm =1$. We consider 
orthogonal polynomials with respect to the inner product $\la \cdot,\cdot\ra_\sw$ defined in \eqref{eq:ipd-w}.
A polynomials $p_n \in \Pi_n(\Omega)$ is an orthogonal polynomial if $\la p_n, q\ra_\sw =0$ for all 
$q\in \Pi_{n-1}(\Omega)$. 
$\CV_n(\Omega,\sw)$ be the space of orthogonal polynomials of degree $n$. If $\Omega$ is a quadratic 
surface, such as the unit sphere $\sph$, then 
\begin{equation*} 
  \dim \CV_n(\Omega,\sw) = \binom{n+d-2}{n} + \binom{n+d-3}{n-1}, \quad n = 1,2,3,\ldots, 
\end{equation*}
where we assume $\binom{n}{k} =0$ if $k < 0$. If $\Omega$ is a solid domain, such as the unit ball $\BB^d$,
then 
\begin{equation*} 
   \dim \CV_n(\Omega,\sw) = \binom{n+d-1}{n}, \quad n = 0, 1,2,\ldots ,
\end{equation*}
The Fourier orthogonal expansion of $f \in L^2(\Omega,\sw)$ is defined as in \eqref{eq:FourierOP} and the
reproducing kernel $P_n(\sw; \cdot, \cdot)$ of $\CV_n(\Omega,\sw)$ is given by \eqref{eq:proj} and 
\eqref{eq:Pn-kernel}. 

We need to consider specific domain $\Omega$ and weight function $\sw$. Our quintessential example is the 
homogeneous space on the unit sphere with respect to the surface measure, for which orthogonal polynomials 
are spherical harmonics. As an extension of the addition formula of spherical harmonics given in 
\eqref{eq:additionS}, we shall assume that that the homogeneous space $(\Omega, \sw, \sd)$ admits a 
closed formula given as an integral of the Jacobi polynomial $P_n^{(\a,\b)}$. The precise definition is 
given in the next subsection. First, we recall the definition of the Jacobi polynomials. These are orthogonal 
polynomials associated with the weight function 
$$
       w_{\a,\b}(t) = (1-t)^\a (1+t)^\b, \qquad \a,\b > -1. 
$$ 
The polynomial $P_n^{(\a,\b)}$ is of degree $n$ and normalized so that $P_n^{(\a,\b)}(1) = \binom{n+\a}{n}$. 
With 
\begin{equation}\label{eq:c_ab}
 c'_{\a,\b} = \frac{1}{2^{\a+\b+1}} c_{\a,\b} \quad\hbox{and} \quad 
   c_{\a,\b} := \frac{\Gamma(\a+\b+2)}{\Gamma(\a+1)\Gamma(\b+1)}.
\end{equation}
the orthogonality of the Jacobi polynomials is given by 
$$
   c_{\a,\b}' \int_{-1}^1 P_n^{(\a,\b)}(t)P_m^{(\a,\b)}(t) w_{\a,\b}(t) \d t = h_n^{(\a,\b)} \delta_{n,m},
$$
where $h_n^{(\a,\b)}$ is the square of the $L^2$ norm that satisfies
$$
  h_n^{(\a,\b)} =  \frac{(\a+1)_n (\b+1)_n(\a+\b+n+1)}{n!(\a+\b+2)_n(\a+\b+2 n+1)}.
$$
For convenience, we introduce the following notation 
\begin{equation}\label{eq:Znab}
         Z_n^{(\a,\b)}(t) := \frac{P_n^{(\a,\b)}(1) P_n^{(\a,\b)}(t)}{h_n^{(\a,\b)}}. 
\end{equation}
The Gegenbauer polynomials $C_n^\l$ are orthogonal with respect to the weight function 
$$
       w_{\l}(t) = (1-t^2)^{\l-\f12},  \qquad \l > -\f12, 
$$ 
which is a special case of the Jacobi polynomials since $w_\l(t) = w_{\l-\f12,\l-\f12}(t)$, but it is normalized 
so that $C_n^\l(1) = \frac{(2\l)_n}{n!}$; more precisely \cite[(4.7.1)]{Sz},
$$
      C_n^\l(t) = \frac{(2\l)_n}{(\l+\f12)_n} P_n^{(\l-\f12,\l-\f12)}(t). 
$$
The orthogonality of the Gegenbauer polynomials is given explicitly by 
$$
   c_{\l} \int_{-1}^1 C_n^{\l}(t)C_m^{\l}(t) w_\l(t) \d t = h_n^{\l} \delta_{n,m}, \qquad 
           h_n^\l = \frac{\l}{n+\l} C_n^\l(1),
$$
where $c_\l = c_{\l-\f12,\l-\f12}$. In particular, it follows readily that 
$$
  Z_n^\l(t):= \frac{n+\l}{\l} C_n^\l(t) = Z_n^{(\l-\f12,\l-\f12)}(t). 
$$

\subsection{Addition formula and convolution structure}
We work in the setting that $\Omega$ is a compact domain and the weight function $\sw$ on $\Omega$ admits 
an addition formula for the the reproducing kernel $P_n(\sw; \cdot,\cdot)$. The addition formula and the convolution 
structure it leads to are defined in \cite{X21}, which we review in this subsection. The following definition appears
in \cite[Definition 3.2]{X21}, in which $Z_n^{(\a,\b)}$ is defined in \eqref{eq:Znab}. 

\begin{defn}\label{defn:additionF}
Let $\sw$ be a weight function on $\Omega$. The reproducing kernel $P_n(\sw; \cdot,\cdot)$ 
is said to satisfy an addition formula if, for some $\a \ge \b \ge - \f12$,
\begin{equation} \label{eq:7Pn}
   P_n (\sw; x, y) = \int_{[-1,1]^m} Z_n^{(\a,\b)} \big(\xi(x, y; u) \big) \d \tau (u), 
\end{equation}
where $m$ is a positive integer; $\xi(x, y; u)$ is a function of $u \in [-1,1]^m$, symmetric 
in $x$ and $y$, and $\xi(x, y; u) \in [-1,1]$; moreover, $\d \tau$ is a probability measure on 
$[-1,1]^m$, which can degenerate to have a finite support. 
\end{defn}

The classical addition formula for spherical harmonics is the degenerate case. The addition formula for
product type weight functions on the unit sphere and the Gegenbauer weight function on the unit ball is 
given in terms of $Z_n^\l$ \cite{DX, X99}, while those on the simplex and on the conic domains are given 
in terms of $Z_n^{(\l-\f12,-\f12)}$ \cite{X20a, X20b}. 

For $1 \le p \le \infty$, we denote by $\|f\|_{p,\sw}$ the $L^p$ norm of $L^p(\Omega, \sw)$ for
$1 \le p < \infty$ and the uniform norm of $C(\Omega)$ for $p = \infty$. We also identify $L^p([-1,1],w_{\a,\b})$
with $C([-1,1])$ when $p = \infty$. 

Making use of the one-dimensional structure premised by the addition formula, a convolution operator is
defined in \cite[Definition 3.3]{X21}. 

\begin{defn}\label{defn:7convol}
Assume $\sw$ admits the addition formula. For $f \in L^1(\Omega, \sw)$ and 
$g \in L^1([-1,1],w_{\a,\b})$, we define the convolution of $f$ and $g$ by 
$$
  (f \ast_\sw  g)(x) :=   \int_{\Omega} f(y) T^{(\a,\b)} g (x,y) \sw(y) \d \sm(y),
$$
where the operator $g\mapsto T^{(\a,\b)} g$ is defined by 
\begin{align*} 
   T^{(\a,\b)} g(x,y) := \int_{[-1,1]^m}  g \big( \xi (x, y; u)\big)  \d \tau(u).
\end{align*}
\end{defn}

The definition of this convolution structure is motivated by the expression of the reproducing kernel specified 
by the addition formula, so that we can write 
\begin{equation}\label{eq:proj=f*g}
     \proj_n(\sw; f, x) = f*_\sw Z_n^{(\a,\b)} \quad\hbox{and}\quad P_n(\sw; x,y) = T^{(\a,\b)} \left( Z_n^{(\a,\b)}\right).
\end{equation}
The operator $T^{(\a,\b)}$ is also defined with a more generic orthogonal polynomial in place of 
$P_n^{(\a,\b)}$ in \cite{X20b}, where it is shown to be bounded \cite[Lemma 6.3]{X20b}. 


\begin{lem}
For $g \in L^p([-1,1],w_{\a,\b})$, $1\le p \le \infty$, and $x \in \Omega$, 
\begin{equation*}
  \left \| T^{(\a,\b)} g (x, \cdot)\right\|_{\sw, p} \le \|g \|_{L^p([-1,1],w_{\a,\b})}.
\end{equation*}
\end{lem}

The convolution operator $f \ast_\sw g$ satisfies the usual Young's inequality. 
Let $p,q,r \ge 1$ and $p^{-1} = r^{-1}+q^{-1}-1$. For $f \in L^q(\Omega, \sw)$ and
$g \in L^r([-1,1]; w_{\a,\b})$, 
\begin{equation} \label{eq:Young}
  \|f \ast_\sw g\|_{L^p (\Omega, \sw)} \le \|f\|_{L^q(\Omega,\sw)}\|g\|_{L^r( [-1,1];\sw)}.
\end{equation}

There is another way to define the convolution structure that uses a sort of translation operator defined 
as in \cite[Definition 3.7]{X21}. 

\begin{defn}
Assume the addition formula holds for the weight function $\sw$. For $0\le \t \le \pi$, the translation 
operator $S_{\t,\sw}$ is defined by 
\begin{equation}\label{eq:Stheta}
    \proj_n (\sw; S_{\t, \sw} f) = \frac{P_n^{(\a,\b)}(\cos \t)}{P_n^{(\a,\b)}(1)} \proj_n(\sw; f), 
      \quad n = 0,1,2,\ldots.
\end{equation}
\end{defn}

Since $|P_n^{(\a,\b)}(\cos \t)| \le |P_n^{(\a,\b)}(1)|$, the operator $S_{\t,\sw} f$ is well defined for
$f\in L^2(\Omega, \sw)$ and, by density, for all $f\in L^1(\Omega, \sw)$. The following proposition
is established in \cite[Proposition 3.8]{X21}. 

\begin{prop}\label{prop:Stheta}
The operator $S_{\t,\sw}$ satisfies the following properties:
\begin{enumerate}[    \rm (i)]
\item For $f \in L^2(\Omega, \sw)$ and $g \in L^1([-1,1], w_{\a,\b})$, 
$$
   (f*_\sw g)(x) = c_{\a,\b} \int_0^\pi S_{\t, \sw} f(x) g(\cos\t) w_{\a,\b}(\cos \t) \sin \t \d \t. 
$$
\item $S_{\t, \sw} f$ preserves positivity; that is, $S_{\t, \sw}  f \ge 0$ if $f \ge 0$. 
\item For $f\in L^p(\sw; \Omega)$, if $1 \le p \le \infty$, or $f \in C(\Omega)$ if $p =\infty$, 
$$
  \|S_{\t, \sw} f \|_{p,\sw} \le \|f\|_{p, \sw} \quad \hbox{and} \quad \lim_{\t\to 0} \|S_{\t, \sw}  f - f\|_{\sw,p} =0.
$$
\end{enumerate}
\end{prop}

These results, under the existence assumption of the addition formula, allow us to deduce many results
of the Fourier orthogonal series on $\Omega$ from the corresponding results of the Fourier-Jacobi series. 
To illustrate the result, we consider the Poisson sum and the Ces\`aro means below. 

\subsection{Poisson integral} 
Let $q_r^{(\a,\b)}(t,s)$ denote the Poisson kernel of the Jacobi polynomials,
$$
q_r^{(\a,\b)}(t,s) =  \sum_{n=0}^\infty \frac{P_n^{(\a,\b)}(t) P_n^{(\a,\b)}(s)}{h_n^{(\a,b)}} r^n, \quad 0 \le r < 1.
$$
For $f \in L^1(\Omega, \sw)$ and $0 \le r <1$, the Poisson integral of $f$ is defined by 
\begin{align} \label{eq:Poisson}
   Q_r(\sw; f,x)\, & = \sum_{n=0}^\infty \proj_n (\sw; f, x) r^n \\ 
      & = f *_\sw q_r^{(\a,\b)} \quad \hbox{with} \quad q_r^{(\a,\b)}(t):= q_r^{(\a,\b)}(t,1). \notag
\end{align}
 
\begin{lem} 
For $\a, \b \ge -\f12$, the kernel $q_r^{(\a,\b)}$ satisfies 
\begin{equation}\label{eq:q_rPoisson}
  q_r^{(\a,\b)}(t) =  c_\b \int_{-1}^1 \frac{1-r }{\left(1-2 \sqrt{r} \sqrt{\f{1+t}{2}}  u + r\right)^{\a+\b+2}}(1-u^2)^{\b-\f12} \d u,
\end{equation}
where $c_\b = c_{\b-\f12,\b - \f12}$ and the identity holds for $\b \to -\f12$ under the limit
\begin{equation}\label{eq:limitInt}
 \lim_{\b\to -\f12} c_\b \int_{-1}^1 g(u)(1-u^2)^{\b-\f12} \d u = \frac{g(1)+g(-1)}{2}.
\end{equation}
In particular, $q_r^{(\a,\b)}(t) \ge 0$ for $t \in [-1,1]$. 
\end{lem} 

\begin{proof}
By \eqref{eq:proj=f*g}, the function $q_r^{(\a,\b)}$ is given by 
$$
 q_r^{(\a,\b)}(t)  = \sum_{n=0}^\infty  Z_n^{(\a,\b)}(t)  r^n.
$$
The polynomial $P_n^{(\a,\b)}(\cos 2 \t)$ can be identified with the generalized Gegenbauer polynomial
of degree $2n$ that are orthogonal with respect to $|x|^{\a+\f12} (1-x^2)^\b$ on $[-1,1]$, and the latter
one can be written as an integral of the Gegenbauer polynomial of degree $2n$. In particular, by 
\cite[Theorem 1.5.6]{DX}, the following identity holds,
$$
  Z_n^{(\a,\b)} (\cos 2 \t) = c_{\b}\int_{-1}^1 Z_{2n}^{\a+\b+1} (u \cos \t) (1-u^2)^{\b-\f12} \d u,
$$ 
where $c_\b = c_{\b-\f12,\b - \f12}$, and the formula holds under the limit \eqref{eq:limitInt} for $\b = - \f12$, 
\begin{equation} \label{eq:Z-quadr-trans}
 Z_n^{(\a,-\f12)} (\cos 2 \t) =  Z_{2n}^{\a+\f12} (\cos \t). 
\end{equation}
The right-hand side of the last integral is zero if we replace $2n$ by $2n+1$, as can be seen by a change of variable 
$u \mapsto -u$ and using $C_n^\l (t) = (-1)^\l C_n^\l(t)$. Hence, using the 
identity
$$
      \sum_{n=0}^\infty Z_n^\l(u) r^n = \frac{1-r^2}{(1-2r u + r^2)^{\l+1}}, \qquad 0 \le r < 1, 
$$
we deduce that
\begin{align*}
  \sum_{n=0}^\infty Z_n^{(\a,\b)} (\cos 2 \t)  r^n \, & =  c_{\b}\int_{-1}^1  \sum_{n=0}^\infty Z_n^{\a+\b+1} (u \cos \t)  (\sqrt{r})^n
  (1-u^2)^{\b-\f12} \d u \\
      & = c_{\b}\int_{-1}^1 \frac{1-r }{(1-2 \sqrt{r} u \cos \t  + r)^{\a+\b+2}} (1-u^2)^{\b-\f12} \d u. 
\end{align*}
Setting $t = \cos (2\t) = 2 \cos^2\t -1$ proves the result for $\b > -\f12$. 
\end{proof}

The basic properties of the Poisson integral is given in the following theorem. 

\begin{thm} \label{thm:poisson}
Assume the addition formula holds for the weight function $\sw$. Then 
\begin{enumerate}
\item  $Q_r(\sw; f_0) = f_0$ for $f_0(x) =1$; 
\item  $Q_r(\sw)$ are positive: $Q_r(\sw; f) \ge 0$ if $f \ge 0$; 
\item  $Q_r(\sw)$ are symmetric: $Q_r(\sw)$ is self-adjoint on $L^2(\Omega,\sw)$;
\item  $Q_r(\sw)$ are contractions on $L^p(\Omega, \sw)$: 
          $\|Q_r(\sw; f)\|_{p,\omega} \le \|f\|_{p,\Omega}$, $1 \le p \le \infty$;
\item If $f \in L^p(\Omega, \sw)$, $1 \le p < \infty$, or $f\in C(\Omega)$ if $p = \infty$, then 
$$
    \lim_{r \to 1-}  \|Q_r(\sw; f) - f \|_{\sw, p} = 0.  
$$
\end{enumerate}
\end{thm}

\begin{proof}
The first three properties follow immediately from \eqref{eq:Poisson} and the properties of $q_r^{(\a,\b)}$.
The fourth one is a consequence of the inequality \eqref{eq:Young} with $r =1$.

We now prove the fifth one. Since $c_{\a,\b}' \int_{-1}^1 q_r^{(\a,\b)} (t) w_{\a,\b}(t) \d t =1$, we have
$$
 Q_r(\sw; f, x) - f(x) = c_{\a,\b} \int_0^\pi \left[S_{\t, \sw} f(x) - f(x)\right] q_r^{(\a,\b)}(\cos\t) w_{\a,\b}(\cos \t) \sin \t \d \t.
$$
By Proposition \ref{prop:Stheta}, for every $\ve > 0$, there is a $\delta > 0$ such that 
$ \|S_{\t,\sw} f -f \|_{\omega,p}  < \ve$ whenever $\t < \delta$.  Hence, by the Minkowski inequality, we obtain
\begin{align*}
   \|Q_r(\sw; f) f- f\|_{p,\sw}\,& \le c_{\a,\b} \int_0^\pi  \| S_{\t, \sw} f- f\|_{p,\sw} 
        q_r^{(\a,\b)} (\cos\t) w_{\a,\b}(\cos \t) \sin \t \d \t \\
        & \le \ve + 2 \|f\|_{p,\sw} c_{\a,\b} \int_\delta^\pi q_r^{(\a,\b)} (\cos\t) w_{\a,\b}(\cos \t) \sin \t \d \t.
\end{align*}
For $0 \le \t \le \pi$, $\cos \tfrac \t 2 \ge 0$. Using the inequality 
\begin{align}\label{eq:poisson-lwbd}
 1-2\sqrt{r} \cos\tfrac \t 2 u + r \, & = (1-\sqrt{r})^2 + 2 \sqrt{r} (1- \cos \tfrac \t 2) + 2 \sqrt{r}\cos \tfrac \t 2 (1- u)  \\
      & \ge (1-\sqrt{r})^2 + 4\sqrt{r} \sin^2 \tfrac{\t}{4}, \notag
\end{align}
it follows readily that, for $\t \ge \delta$, 
\begin{align*}
  q_r^{(\a,\b)} (\cos\t) \, & = c_\b \int_{-1}^1 \frac{1-r }{\left(1-2 \sqrt{r} \cos\f{\t}{2} u + r\right)^{\a+\b+2}}(1-u^2)^{\b-\f12} \d u \\
     & \le \frac{1-r }{\left((1-\sqrt{r})^2 + 4\sqrt{r} \sin^2 \tfrac{\delta}{4}\right)^{\a+\b+2}}
\end{align*} 
so that $q_r^{(\a,\b)}(\cos \t) \to 0$ as $r \to 1-$ for $\delta \le \t \le \pi$. Putting these together completes the proof. 
\end{proof}

The Poisson operator will be needed in the Subsection 2.5.

\subsection{Ces\`aro means and the Hardy-Littlewood maximal function}
For $\delta > 0$, the Ces\`aro $(C,\delta)$ means of the  Fourier orthogonal series with respect to $\sw$ on 
$\Omega$ are defined by
\begin{align*}
 S_n^\delta (\sw;f,x) :=  \f{1}{\binom{n+\delta}{n}} \sum_{k=0}^n \binom{n-k+\delta}{n-k} \proj_k(\sw; f,x).
\end{align*}
Let $k_n^{(\a,\b), \delta}(t,s)$ be the kernel of the $(C,\delta)$ mean of the Fourier-Jacobi series. Then
$$
 k_n^{(\a,\b), \delta}(t): = k_n^{(\a,\b), \delta}(t,1) = \f{1}{\binom{n+\delta}{n}} \sum_{k=0}^n \binom{n-k+\delta}{n-k} 
      Z_k^{(\a,\b)}(t).
$$
Using \eqref{eq:proj=f*g}, we can write $S_n^\delta (\sw;f)$ as the convolution 
\begin{equation} \label{eq:(C,delta)}
      S_n^\delta(\sw; f) = f *_\sw  k_n^{(\a,\b), \delta}.
\end{equation}
Our method allows us to derive the properties of $S_n^\delta (\sw;f)$ from those of $k_n^{(\a,\b),\delta}$. 
In particular, the following theorem is stated in \cite[Theorem 3.6]{X21}. 
 
\begin{thm}\label{thm:cesaro}
Let $\sw$ be a weight function that admits the addition formula. The Ces\`aro $(C,\delta)$ means 
of the Fourier orthogonal series with respect to $\sw$ satisfy 
\begin{enumerate}
\item If $\delta \ge \a+\b+2$, then $S_n^\delta(\sw; f)$ is a nonnegative operator;
\item If $\delta > \max\{\a,\b\} + \f12$, then for $n= 0,1,2,\ldots$,
$$
    \left \|S_n^\delta(\sw; f)\right \|_{p,\sw} \le  \|f\|_{p,\sw}, \qquad 1 \le p \le \infty.
$$
\end{enumerate}
\end{thm} 

For the unit sphere, the Ces\`aro means of the Fourier expansions in spherical harmonics plays an 
essential role for the proof of a multiplier theorem, as shown in \cite{BC}. Moreover, the approach also
applies to $h$-spherical harmonics series on the unit sphere, which are orthogonal polynomials with 
respect to a family of product weight functions on the sphere. It turns out that the same approach can
also be used to derive an multiplier theorem when $(\Omega,\sw)$ admits an addition formula, 
provided that the Ces\`aro means satisfies an additional property. 

To state this property, we defined the maximal Ces\`aro $(C,\delta)$ operator: for $\delta \ge 0$, 
$$
  S_*^\delta(\sw; f, x):= \sup_{n \ge 0} \left|S_n^\delta(\sw; f, x)\right|, \qquad x \in \Omega.
$$
We also need the Hardy-Littlerwood maximal function on the homogeneous space $(\Omega,\sw,\sd)$,
which is defined by 
$$
   M_\sw f(x): = \sup_{r > 0} \frac{1}{\sw(B(x,r))} \int_{B(x,r)} |f(y)| \sw(y) \d \sm(y), \quad x \in \Omega.
$$ 
The property states that the maximal $(C,\delta)$ operator is bounded by the maximal operator, which we
state as an assertion. 

\medskip
\textbf{Assertion 1}. {\it Assume $\sw$ admits the addition formula \eqref{eq:7Pn}. Let $\delta_0= \max\{\a,\b\}+\f12$.
For $\delta > \delta_0$ and $f\in L^1(\Omega,\sw)$, 
\begin{equation} \label{eq:maxS<M}
      S_*^\delta (\sw; f, x)  \le c M_\sw f(x), \qquad x \in \Omega,
\end{equation}
where $M_\sw f$ is the Hardy-Littlewood maximal function. }
\medskip

The maximal function $M_\sw$ requires local information for $f$, specified by the metric $\sd$ on $\Omega$.
This is the first time that we encounter the metric, since our development so far is entirely based on the 
orthogonal structure specified by the addition formula \eqref{eq:7Pn}. In order to establish the inequality 
\eqref{eq:maxS<M}, we shall need further information on the orthogonal structure, so that the convolution 
structure $f *_\sw g$ can be localized. Merely the existence of an addition formula \eqref{eq:7Pn} is not
sufficient for the inequality \eqref{eq:maxS<M}. The Assertion 1 is known to be true for the unit sphere, the 
unit ball and the simplex with classical weight functions. We shall prove that it holds on conic domains as well 
in the next section. 

The Hardy-Littlewood inequality satisfies the Fefferman-Stein inequality \cite{FS}. A proof is given 
in \cite[pp. 51-55]{Stein}. 

\begin{thm} \label{thm:FeSt}
If $1 < p, q < \infty$ and $\{f_j\}$ is a sequence of functions on $\Omega$, then 
$$
   \left\| \left( \sum_{k=0}^\infty \left|M_\sw f \right |^q \right)^{\f1q}\right \|_{p,\sw} \le 
      c_{p,q}  \left \| \left( \sum_{k=0}^\infty  \left | f_k \right |^q \right)^{\f1q}\right \|_{p,\sw}. 
$$
\end{thm}

As a consequence of the Assertion 1, we immediately obtain the following corollary. 

\begin{cor}\label{cor:Ces-FeSt}
Assume Assertion 1 holds and $\delta > \delta_0$. Then, for any sequence $\{n_k\}$ of positive integers 
and $f_k \in L^p(\Omega, \sw)$, $1 < p < \infty$,  
$$
   \left\| \left( \sum_{k=0}^\infty \left|S_{n_k}^\delta (\sw, f_k)\right |^2 \right)^{\f12}\right \|_{p,\sw} \le 
      c  \left \| \left( \sum_{k=0}^\infty  \left | f_k \right |^2 \right)^{\f12}\right \|_{p,\sw}. 
$$
\end{cor}

Both the Fefferman-Stein theorem and the corollary will be needed for proving the multiplier theorem,
as we shall describe in the next subsection. 

\subsection{Multiplier theorem}

We now state the Marcinkiewicz multiplier theorem for the Fourier orthogonal series, for which we need 
the difference operator. Let $\{a_k\}_{k=0}^\infty$ be a sequence of real numbers. The difference operators 
 $\triangle^n$, $n = 0,1,2,\ldots$, are defined recursively by 
$$
   \triangle^0 a_k = a_k, \quad  \triangle^1 a_k = a_k - a_{k+1}, \quad   \triangle^{n+1} a_k =   \triangle^n (  \triangle a_k), \quad
    n=1,2, \ldots. 
$$

\begin{thm}\label{thm:multiplier} 
Assume that $(\Omega,\sw)$ admits the addition formula \eqref{eq:7Pn} and that the Assertion 1 holds. 
Let $\{\mu_j\}_{j=0}^\infty$ be a bounded sequence of real numbers such that 
$$
  \sup_{j \ge 0} 2^{j(k-1)} \sum_{\ell =2^j +1}^{2^{j+1}} \left|\triangle^k \mu_{\ell}\right| \le M < \infty
$$
for some positive integer $k \ge \lfloor \delta_0 +1\rfloor$. Then $\{\mu_j\}$ defines an $L^p(\Omega,\sw)$,
$1 < p < \infty$, multiplier; that is
$$
   \left \| \sum_{k =0}^\infty \mu_k \proj_k (\sw; f)\right \|_{p,\sw} \le c  \|f\|_{p,\sw}, \quad 1 < p < \infty, 
$$
where $c$ is independent of $\mu_j$ and $f$. 
\end{thm}

For the unit sphere $\sph$ with $\sw(x) = 1$, such a multiplier theorem is established in \cite{BC}. The proof 
in \cite{BC} is based on the Littlewood-Paley theory on the sphere but and it is quite involved and original. 
The approach also applies, with appropriate twist, on the unit sphere with product type weight functions 
and, as a consequence, on the unit ball and the simplex with classical weight function \cite{DaiX2}. Moreover, 
the proof can be extended to our general setting, under Assertion 1, almost straightforwardly. We provide 
an outline of the proof below to show that this is indeed the case. Throughout the rest of the section, we 
assume that $(\Omega,\sw)$ admits the addition formula \eqref{eq:7Pn} and the Assertion~1 holds. 

The proof of the multiplier theorem relies on several auxiliary $g$-functions. The first one is defined in terms 
of the Poisson integral $Q_r(\sw; f)$, 
$$
  g(f) := \left( \int_0^1 (1-r) \left | \frac{\d}{\d r} Q_r(\sw; f) \right |^2 \d r \right)^{\f12},
$$ 
and the function $x \mapsto g(f)(x)$, defined on $\Omega$, is bounded in $L^p(\Omega,\sw)$. 

\begin{prop}
If $1 < p < \infty$ and $f \in L^p(\Omega,\sw)$, then
\begin{equation}\label{eq:g(f)bound}
   c_p^{-1}  \|f\|_{p,\sw} \le  \|g(f)\|_{p,\sw} \le c_p  \|f\|_{p,\sw},
\end{equation}
where the condition $\int_{\Omega} f(x)\sw(x) \d \sm(x) =0$ is required in the left-hand inequality. 
\end{prop}

\begin{proof}
Let $T^t$ be the operator defined by $T^t f = Q_{e^{-t}}(\sw; f)$. By \eqref{eq:Poisson}, it is easy
to see that $T^t$ satisfies $T^{t_1+t_2} = T^{t_1}T^{t_2}$. By the first four properties of 
Theorem \ref{thm:poisson}, $\{T^t\}$ is a symmetric semigroup of operators. Hence, by 
\cite[Theorem 10, p.111]{Stein}, the Littlewood-Paley $g$-function 
$$
    \tilde g(f):=  \left( \int_0^\infty \left | \frac{\d}{\d t}T^t f \right |^2 t \d t \right)^{\f12}
        =  \left( \int_0^1  \left | \frac{\d}{\d r} Q_r(\sw; f) \right |^2 r |\log r| \d r \right)^{\f12}
$$
satisfies $c_p^{-1} \|f\|_{p,\sw} \le \|\tilde g(f) \|_{p,\sw} \le c_p \|f\|_{p,\sw}$. Since $r |\log r| \le c(1-r)$ on $[0,1]$, 
we obtain $\tilde g(f) (x) \le g(f)(x)$, which implies immediately the left-hand inequality in \eqref{eq:g(f)bound}. 
To prove the right-hand inequality, we split the integral of $g(f)$ into two parts, so that $g(f) \le g_1(f) + g_2(f)$, 
where $g_1(f)$ has the integral over $[0,\f12]$ and $g_2(f)$ has the integral over $[\f12,1]$. For $g_1(f)$, we use 
$\|\proj_n (\sw; f)\|_{p, \sw} \le Z_n^{(\a,\b)}(1) \|f\|_{\sw, p} \le c n^{\a+1} \|f\|_{\sw, p}$, which follows from 
\eqref{eq:proj=f*g} and \eqref{eq:Young}, to obtain
$$
  \|g_1(f)|_{p,\sw} \le \sup_{0 < r \le \f12} \left \| \frac{\d}{\d r} Q_r(\sw; f) \right \|_{p,\sw} 
    \le \sum_{n=1}^\infty n 2^{-n+1} \|\proj_n (\sw; f)\|_{p, \sw} 
    \le c \| f\|_{p, \sw}.
$$ 
For $g_2(f)$, we observe $r|\log r| \sim 1-r$ on $[\f12,1]$, so that $g_2(f)(x) \le \tilde g(f) (x)$ and, consequently, 
the boundedness of $g_2$ follows from that of $\tilde g(f)$. 
\end{proof}

The above proof is a straightforward extension of the version for the unit sphere in \cite{BC}; see \cite[Thm. 3.2.1]{DaiX}
since the weight and the domain does not play an essential role in the proof. This will be the case for the most part
in the proof of Theorem \ref{thm:multiplier} below and we shall be brief.

To simplify the notation, we shall use $Q_r f = Q_r(\sw; f)$ and $S_n^\delta f = S_n^\delta(\sw; f)$ for the Poisson 
integral and the $(C,\delta)$ means in the rest of this subsection. Using the orthogonality, it is easy to verify that 
$Q_r S_n^\delta = S_n^\delta Q_r$ for $n \ge 0$. We will need several identities on these means that hold in fact
for the the Poisson sum and the $(C,\delta)$ means of any given sequence. We call such identities algebraic and 
take them as granted. All these identities can be found in \cite{BC, DaiX}. The first such identity is 
$$
 \frac{\d}{\d r} Q_r f = - (1-r)^\delta \sum_{n=1}^\infty (n+\delta+1)\binom{n+\delta}{n} 
  \left ( S_n^{\delta +1}f - S_n^\delta f \right)r^{n-1};
$$
see, for example, \cite[(3.2.4) and (3.2.5)]{DaiX}. This identity implies that $g(f)$ is related to another $g$-function
defined in terms of the $(C, \delta)$ means. For $\delta \ge 0$, define
$$
  g_\delta (f) := \left(\sum_{n=1}^\infty  \frac{1}{n} \left|S_{n}^{\delta+1}(\sw; f) - S_n^\delta(\sw; f)\right|^2\right)^{\f12}.
$$ 
Then the aforementioned identity implies that $g(f) \le c g_\delta(f)$ and, consequently, the proposition below. 

\begin{prop} \label{prop:gd-g}
If $f\in L^p(\Omega,\sw)$ satisfying $\int_\Omega f(x) \sw(x) \d \sm(x) =0$, then
\begin{equation} \label{eq:gd-g}
  \|f\|_{p,\sw} \le c \|g_\delta(f)\|_{p,\sw}, \qquad 1 < p < \infty. 
\end{equation}
\end{prop} 

The inverse inequality of \eqref{prop:gd-g} also holds and it can be established for a perturbation of $g_\delta$.
Let $\{\nu_k\}_{k=1}^\infty$ be a sequence of positive numbers satisfying 
\begin{equation}\label{eq:sum_nu}
     A : = \sup_{n \ge 1} \frac{1}{n} \sum_{k=1}^n \nu_k < \infty. 
\end{equation}
The perturbed $g_\delta$ function is denoted by $g_\delta^*$ and defined by 
$$
 g_\delta^*(f):= \left(\sum_{n=1}^\infty \frac{\nu_n}{n}\left|S_{n}^{\delta+1}(\sw; f) - S_n^\delta(\sw; f)\right|^2\right)^{\f12}.
$$
Evidently, $g_\delta^*(f) = g_\delta(f)$ when $\nu_k=1$ for all $k$. We need the boundedness of $g_\delta^*$. 

\begin{prop}\label{prop:gd*bd}
Assume that Assertion 1 holds for all $\{n_k\} \in \NN$ and $f_k \in L^p(\Omega,\sw)$. Then
\begin{equation}\label{eq:gd*bd}
\|g_\delta^*(f)\|_{p,\sw} \le c_p\|f\|_{p,\sw}, \qquad 1 < p < \infty,
\end{equation}
where $c_p$ is a constant independent of $f$ and $\{\nu_k\}$. 
\end{prop} 

The proof of this proposition is long. We trace its steps below and indicate where Assertion 1 is needed. 
Several algebraic identities that involve $Q_r f$ and $S_n^\delta f$ will be need. The first one is 
(\cite{BC} and \cite[(3.2.11)]{DaiX})
$$
  S_n^\delta(Q_r f) = Q_r(S_n^\delta f) = \sum_{j=0}^n b_{j,n}^\delta(r) S_j^\delta f, \qquad \hbox{where} \quad 
    \sum_{j=0}^n |b_{j,n}^\delta(r)| \le c_\delta, 
$$
and $b_{j,n}^\delta(r)$ is a polynomial in $r$ that can be given explicitly, from which the estimate of the sum
follows. By Assertion 1 and Corollary \ref{cor:Ces-FeSt}, this identity allows us to conclude 
\begin{equation}\label{eq:SkQr1}
  \left \| \left( \sum_{k=1}^\infty \left| S_k^\delta Q_{r_k} f_k \right|^2\right)^{\f12} \right\|_{p,\sw} 
    \le c_p  \left \| \left( \sum_{k=1}^\infty  \left| f_k \right|^2 \right)^{\f12} \right\|_{p,\sw}, 
\end{equation}
where $r_k \in (0,1)$. Let $I_k$ be a subinterval of $[r_k,1)$ and let $|I_k|$ be the length of $I_k$. 
Define $\{r_{k,i} : 0 \le i \le 2^n\}$ by $r_{k,i} - r_{k,i-1} = 2^{-n} | I_k|$. Since the Poisson sum satsfies 
$Q_{r_k} f = Q_{r_k/r_{k,i}}Q_{r_{k,i}} f$, we can write
$$
    \left |S_k^\delta Q_{r_k} f_k \right|^2 = \frac{1}{2^n} \sum_{i=1}^{2^n} \left| S_k^\delta Q_{r_k/r_{k,i}} (Q_{r_{k,i}} f)\right|^2.
$$
Using this identity first and then applying \eqref{eq:SkQr1}, we conclude that 
$$
  \left \| \left( \sum_{k=1}^\infty \left| S_k^\delta Q_{r_k} f_k \right|^2\right)^{\f12} \right\|_{p,\sw} 
   \le  c_p  \left \|  \sum_{k=1}^\infty  \left(  \frac{1}{2^n} \sum_{i=1}^{2^n} 
        \left| Q_{r_{k,i}} f \right|^2 \right)^{\f12} \right\|_{p,\sw}.  
$$
The inner sum in the right-hand side is a Riemann sum of the integral $\frac{1}{|I_k|} \int_{I_k} |Q_{r_k} f|^2 \d r$.
Hence, passing to the limit $n\to \infty$, we obtain the following lemma. 

\begin{lem}
Assume that $1 < p < \infty$ and Assertion 1 holds. If $r_j \in (0,1)$ and $I_j$ is a subinterval of $[r_j,1)$ 
for $j= 1,2\ldots$, then the inequality 
\begin{equation}\label{eq:lemSQ}
  \left \| \left( \sum_{k=1}^\infty \left| S_k^\delta Q_{r_k} f_k \right|^2\right)^{\f12} \right\|_{p,\sw} 
    \le c_p \left \| \left( \sum_{k=1}^\infty \frac{1}{|I_k|}\int_{I_k} \left| Q_{r} f_k \right|^2 \d r \right)^{\f12} \right\|_{p,\sw}.
\end{equation}
\end{lem} 

Let $\eta \in C^\infty(\RR)$ be a cut-off function such that $\eta(t) =1$ for $|t| \le 1$ and $\eta(t) = 0$ for
$|t| \ge 2$. For $n =1,2,\ldots$, define 
$$
  L_n f = \sum_{k=0}^{2N} \eta\left(\f k n \right) \proj_k f \quad \hbox{and} \quad
  D_n f = \sum_{k=0}^{2N} \eta\left(\f k n \right)k \proj_k f.
$$
Then $L_n f$ is a resampling of $f$ and $D_n$ is a resampling of $\frac{\d}{\d r} Q_r f$ when $r=1$. The
Assertion 1 is needed to prove the following lemma. 

\begin{lem}
Let $\delta_0$ be the index in Assertion 1. Then
\begin{equation} \label{eq:QrM}
  \left | Q_r (D_N f) \right |  \le c\, M_\sw \left(\f{\d}{\d r} Q_r f\right).
\end{equation}
\end{lem}

\begin{proof}
The operator $L_n f$ is highly localized.  A standard procedure of summation by parts $\lfloor \delta_0 +1\rfloor$ 
times yields
$$
   \sup_{N \ge 0} \left|L_N f(x) \right| \le c  \sup_{N \ge 0} \left |S_N^{\lfloor \delta_0 +1\rfloor}f(x)\right | \le c \, M_\sw f(x)
$$ 
where the last step follows by Assertion 1. The identity $Q_r (D_N f) = r L_N( \f{\d}{\d r} Q_r f)$ can be easily 
verified, from which \eqref{eq:QrM} follows readily. 
\end{proof}

The next algebraic identity allows us to replace $f$ by $Q_r f$ in the proof. It is given by (cf. \cite[(3.2.16)]{DaiX}),
\begin{align}\label{eq:SnQr2}
 S_n^{\delta+1} f - S_n^\delta f = \,&  r^{-n} \left[ S_n^{\delta+1} (Q_r f) - S_n^\delta (Q_rf) \right] \\
    &  + \sum_{j=1}^{n-1} \frac{j+\delta+1}{n+\delta+1} a_{j,n}^\delta(r) \left[ S_j^{\delta+1} (Q_r f) - S_j^\delta (Q_rf) \right],
    \notag
\end{align}
where the coefficients $a_{j,n}^\delta(r)$ satisfy $\displaystyle{\max_{0\le j \le n-1} |a_{j,n}^\delta(r)| \le c (1-r)}$
with $c$ independent of $n$; see \cite[Lemma 3.2.5]{DaiX}. Let $\nu_k$ be given by \eqref{eq:sum_nu}. 
Considering $1+ A^{-1} \nu_j$ instead of $\nu_j$, if necessary, we can 
assume that $n \le \sum_{j=1}^n \nu_j \le 2n$. Let $\mu_1 = 1$ and $\mu_n = 1 + \sum_{j=1}^{n-1} \nu_j$ for $n >1$.
Then $r_n := 1 - \frac1 {\mu_n}$ satisfies $1- \f1 n  \le r_n \le 1- \f1 {2n-1}$. In particular, for $r = r_n$, 
$|a_{j,n}^\delta(r_n)| \le c(1-r_n) \le c n^{-1}$. Hence, by \eqref{eq:SnQr2},
$$
 \left| S_n^{\delta+1} f - S_n^\delta f\right|  \le  c \left | S_n^{\delta+1} f_n - S_n^\delta f_n \right |
     + c n^{-2} \sum_{j=1}^{n-1} j \left | S_j^{\delta+1} f_n - S_j^\delta f_n\right |.
$$
We need a further algebraic identity that states \cite[(3.2.19)]{DaiX}, for $0 \le j \le N$, 
\begin{equation} \label{eq:SnQr3}
   S_j^{\delta+1} (Q_r f) - S_j^\delta (Q_r f) = \frac{-1}{j+\delta+1} r  S_j^\delta  \big(Q_{r} (D_N f)\big).
\end{equation}
Using this identity and the Cauchy-Schwarz inequality, it follows that 
\begin{align*}
 \left| S_n^{\delta+1} f - S_n^\delta f\right|^2  \le c n^{-2} \left| S_n^{\delta} Q_{r_n} (D_N f)\right|^2+ 
    c n^{-3} \sum_{j=1}^{n-1} \left | S_j^{\delta} Q_{r_n} (D_N f) \right |^2. 
\end{align*}
Consequently, applying \eqref{eq:lemSQ}, we obtain
\begin{equation*}
 \left \| \left(\sum_{n=1}^N \frac{\nu_n}{n} \left| S_n^{\delta+1} f - S_n^\delta f \right |^2 \right)^{\f12} \right\|_{p,\sw}
  \le c  \left \| \left(\sum_{n=1}^N \frac{\nu_n}{n^3}
        \frac{1}{|I_n|} \int_{r_n}^{r_{n+1}} \left|Q_r (D_N f) \right |^2 \right)^{\f12} \right\|_{p,\sw},
\end{equation*}
where $|I_n| = r_{n+1} - r_n$. In particular, $\nu_n / |I_n| =  \mu_n \mu_{n+1} \sim n^2$ and $n^{-1} < 1-r$ for 
$r_n < r<r_{n+1}$. Hence, applying \eqref{eq:QrM}, the right-hand side of the displayed inequality is further 
bounded by 
\begin{align*}
& c \left \| \left(\sum_{n=1}^N  \frac{1}{n} \int_{r_n}^{r_{n+1}}  
      \left|M_\sw \left( \f{\sd}{\sd r} Q_r f\right) \right |^2 \d r \right)^{\f12} \right\|_{p,\sw} \\
    & \qquad \qquad \le  c  \left \| \left(\sum_{n=1}^N   \int_{r_n}^{r_{n+1}} (1-r) 
      \left|  \f{\sd}{\sd r} Q_r f \right |^2 \d r \right)^{\f12} \right\|_{p,\sw} 
       \le c \| g(f)\|_{p,\sw},
\end{align*}
where the second step follows from applying the Fefferman-Stein inequality to the Riemann sum of the integral
over $[r_n,r_{n+1}]$. Taking the limit $N\to \infty$ proves the inequality \eqref{eq:gd*bd} and 
Proposition \ref{prop:gd*bd}. In particular, this shows that the proof for the unit sphere carries over to the 
general setting when $(\Omega,\sw)$ admits an addition formula and that the Assertion 1 holds. 

The proof of the multiplier theorem, Theorem \ref{thm:multiplier}, follows form Proposition \ref{prop:gd-g} and 
Proposition \ref{prop:gd*bd} without further complication. Indeed, let $F = \sum_{n=0}^\infty \mu_n \proj_n f$;
then $\|F\| \le g_\delta(F)$ by \eqref{eq:gd-g}, so that it is sufficient to prove 
$$
g_\delta(F) \le g_\delta^*(f) =
    \left(\sum_{n=1}^\infty \frac{\nu_n}{n}\left|S_{n}^{\delta+1}(\sw; f) - S_n^\delta(\sw; f)\right|^2\right)^{\f12}
$$
for some sequence $\{\nu_n\}$ of positive numbers satisfying \eqref{eq:sum_nu}. The proof of this inequality
is entirely algebraic. Hence, Theorem \ref{thm:multiplier} follows as in the case of the unit sphere; see
\cite{BC} and \cite[Section 3.3]{DaiX}.

\subsection{Maximal functions}
The Assertion 1 states that the maximal $(C,\delta)$ operator is bounded by the Hardy-Littlewood maximal
function $M_\sw$. We discuss another maximal function defined via the convolution operator $f*_\sw g$ in
Definition \ref{defn:7convol}, which is useful for studying operators defined via convolution. Let $\chi_E$ 
denote the characteristic function of the set $E$. 

\begin{defn}\label{defn:CMw}
Assume $\sw$ admits the addition formula \eqref{eq:7Pn}. For $f \in L^1(\Omega, \sw)$ and $0 \le \t \le \pi$, define
\begin{equation} \label{eq:CMw}
    \CM_\sw f(x)  := \sup_{0\le \t\le \pi} \frac{ (f*_\sw \chi_{[\cos\t,1]}) (x)} 
         { {\int_0^\t w_{\a,\b}(\cos \phi) \sin\phi \,\d \phi}}.
\end{equation}
\end{defn}

Let $S_{\t,\sw}$ be the generalized translation operator defined in \eqref{eq:Stheta}. By Proposition \ref{prop:Stheta},
we can also write
$$
  \CM_\sw f(x) = \sup_{0\le \t \le \pi} \frac{ \int_0^\t S_{\t, \sw} f(x) w_{\a,\b}(\cos \phi) \sin\phi \,\d \phi}
      {\int_0^\t w_{\a,\b}(\cos \phi) \sin\phi \,\d \phi}.
$$

For the unit sphere $\sph$, let $c(x,\t) = \{y: \la x,y\ra =\cos \t\}$ be the spherical cap and let $\d \s$ denote the 
Lebesgue measure on the sphere. Then the maximal function on $\sph$ satisfies \cite[Lemma 3.2]{DaiX}
$$
    M f(x) = \sup_{0\le \t \le \pi} \frac{1}{c(x,\t)} \int_{c(x,\t)} |f(y)| \d \s(y),
$$
which is exactly the Hardy-Littlewood maximal function $M f$ on the unit sphere. This, however, does not hold
in our general setting. 

While the Hardy-Littlewood maximal function relies on the geometric information of the domain, as seen in
the metric, the definition of the $\CM_\sw$ is purely analytical. The latter, however, is natural for studying 
operators defined by the convolution $*_\sw$ on $(\Omega, \sw)$. 

\begin{thm} \label{thm:f*g<M}
Let $g \in L^1([-1,1],w_{\a,\b})$ with $\a \ge \b \ge -\f12$. Assume $k(\t) = g(\cos \t)$ is a continuous, nonnegative, 
and decreasing function on $[0,\pi]$. Then for $f \in L^1(\Omega, \sw)$,
$$
        |f\ast_\sw g(x)| \le c \CM_\sw f(x), \qquad  x \in \Omega,
$$
where $c = \int_0^\pi k(\phi) \sw(\cos \phi) \sin \phi \d \phi$.
\end{thm}

\begin{proof}
The proof follows from the standard procedure of integration by parts. For fixed $x$, define 
$F_{x} (\t) = \int_0^\t S_{\t,\sw} f(x,t)  \sw (\cos \phi) \sin \phi \d \phi$. Then
\begin{align*}
  f \ast_\sw g(x) \, & =   \int_0^\pi g(\cos \t) S_\t^\l f(x) w_{\a,\b} (\cos \t)\sin \t \d \t \\
          & =   \left[ F_x (\pi) k(\pi) - \int_0^\pi k'(\t) F_x (\t)w_{\a,\b} (\cos \t)\sin \t \d \t \right].
\end{align*}
Since $k(\t)$ is nonnegative and $k'(t) \le 0$, we obtain 
\begin{align*}
 \left| f \ast_\sw g(x)\right| \, & \le  \CM_\sw f(x)  \left[k(\pi) \int_0^\pi k(\phi) w_{\a,\b} (\cos \phi)\sin \phi \d \phi k(\pi) \right.\\
    & \qquad\qquad\qquad \qquad - \left. \int_0^\pi k'(\t) \int_0^\t w_{\a,\b} (\cos \phi)\sin \phi \d \phi \d \t \right] \\
      & \le \CM_\sw f(x) \int_0^\pi k(\phi)  w_{\a,\b} (\cos \phi)\sin \phi \d \phi,
\end{align*}
where the last step follows from another integration by parts. 
\end{proof}

As an example, let us consider the Poisson integral operator. 

\begin{thm}\label{thm:PoissonM}
Assume $\sw$ admits the addition formula \eqref{eq:7Pn} and Assertion 1 holds. Then for $f\in L^1(\Omega; \sw)$ 
and every $x \in \Omega$, 
$$
        \sup_{0< r< 1} \sQ_r(\sw; f, x)  \le c  \CM_\sw f(x).
$$
\end{thm} 

\begin{proof}
By the inequality \eqref{eq:poisson-lwbd}, the kernel $q_r^{(\a,\b)}$ satisfies 
$$
  q_r^{(\a,\b)} (\cos \t) \le \frac{1-r}{\big( (1-\sqrt{r})^2 + \t^2\big)^{\a+\b+1} } =: k(\t) 
$$
It is easy to verify that this $k(\t)$ satisfies the assumption of Theorem \ref{thm:f*g<M}, so that
$$
 \left |\sQ_r(\sw; f, x) \right| = \left |f *_\sw q_r^{(\a,\b)} (x)\right | \le c \CM_\sw(f; x),
$$
where $c$ is a constant independent of $r$. 
\end{proof}

For the Ces\`aro $(C,\delta)$ operator, we obtain the following somewhat weaker result. 

\begin{thm}\label{thm:weak(C,d)bound}
Assume $\sw$ admits the addition formula \eqref{eq:7Pn} and Assertion 1 holds. If $\delta \ge \a+ \b+2$ 
and $f\in L^1(\Omega; \sw)$,  then for every $x \in \Omega$, 
\begin{equation}\label{eq:S*<CM}
       S_*^\delta (\sw_; f, x)  \le c \CM_\sw f(x).
\end{equation}
\end{thm} 

\begin{proof}
By \eqref{eq:(C,delta)}, we need an estimate of $k_n^{(\a,\b),\delta}(u,1)$. For $\delta \ge \a+\b+2$, the
estimate given in \cite[(9.4.4) and (9.41.14)]{Sz} states 
\begin{align*}
  \left| k_n^{(\a,\b),\delta}(\cos \t,1) \right|  \le c  n^{-1} (\t+n^{-1})^{-(2\a+3)} =: k_n(\t).
\end{align*}
Clearly $k_n(\t)$ satisfies the conditions of \eqref{thm:f*g<M} and that $\int_0^\pi k_n(\t) w_{\a,\b}(\cos \t) \sin \t \d \t$
is bounded by a constant independent of $n$ can be easily verified. Hence,
$$
  \left|S_n^\delta \left(\sw_{-1,\g}; f, (x,t)\right)\right|  \le c \CM_\g f(x,t)
$$
for all $n \ge 0$. Taking supreme over $n$ completes the proof. 
\end{proof}

This result is weaker in the case of the unit sphere and the unit ball, for which the condition $\delta \ge \a+\b+2$ 
can be lowered to $\a +\f12$. This, however, requires more explicit formulation of the addition formula in
\eqref{eq:7Pn}. The stronger version holds for the conic domains, as we shall show in the following two sections. 

Since the maximal function $\CM_\sw f$ is not the Hardy-Littlewood maximal function, it is not immediately
clear if it satisfy the usual property of the maximal function. In many cases, however, the two maximal 
functions are closely related. For the conic domains, we shall show that $\CM_\sw f$ is bounded by
the Hardy-Littlewood maximal function. 

\section{Fourier orthogonal series on the conic surface}
\setcounter{equation}{0}

In this section we work with the conic surface in $\RR^{d+1}$ defined, for $d\ge 2$, by 
$$
     \VV_{0}^{d+1}= \{(x,t): \|x\| = t, \, x \in \RR^d, \, 0 \le t \le 1\}.
$$
The addition formula on this domain is established only recently in \cite{X20a}. We discuss the orthogonal 
structure in the first subsection, study the maximal function in the second section, and state the multiplier
theorem in the third subsection. 

\subsection{Analysis on the conic surface}
For $\g \ge -\f12$, let $\sw_{-1,\g}$ be the weight function defined by 
$$
   \sw_{-1,\g}(t) = t^{-1} (1-t)^\g,  \qquad 0 < t <1.  
$$
We consider the homogeneous space $(\VV_0^{d+1}, \sw_{-1,\g}, \sd_{\VV_0})$, where 
$\sd_{\VV_0}$ is the distance function on $\VV_0^{d+1}$, defined by
\begin{equation*}
  \sd_{\VV_0} ((x,t), (y,s)): =  \arccos \left(\sqrt{\frac{\la x,y\ra + t s}{2}} + \sqrt{1-t}\sqrt{1-s}\right)
\end{equation*}
for $(x,t)$ and $(y,s)$ on $\VV_0^{d+1}$. For $\Omega = \VV_0^{d+1}$, the ball becomes the conic cap. 
For $r > 0$ and $(x,t)$ on $\VV_0^{d+1}$, we denote the conic cap centered at $(x,t)$ with radius $r$ by 
$$
      \sc((x,t), r): = \left\{ (y,s) \in \VV_0^{d+1}: \sd_{\VV_0} \big((x,t),(y,s)\big)\le r \right\}.
$$   
Then $\sw_{-1,\g}$ is a doubling measure with respect to $\sd_{\VV_0}$ since it satisfies \cite[(4.4)]{X21}
\begin{align} \label{eq:sw-doubling}
 \sw_{-1,\g} \big(\sc((x,t), r)\big):= \,& \bs_\g \int_{ \sc((x,t), r)} \sw_{-1,\g} (s) \d \s(y,s)  \\
    \sim \, & r^d (t+ r^2)^{\f{d-2}{2}} (1-t+ r^2)^{\g+\f12}, \notag
\end{align}
where $\bs_\g$ is the normalized constant defined by $\bs_\g \int_{\VV_0^{d+1}} \sw_{-1,\g} (t) \d \s (x,t) =1$.

The orthogonality on $L^2(\VV_0^{d+1},\sw_{-1,\g})$ is defined in terms of the inner product 
$$
\la f, g\ra_{\sw} =\bs_{\g} \int_{\VV_0^{d+1}} f(x,t) g(x,t) \sw_{-1,\g}(t)\d \s(x,t),
$$ 
where $\d \s$ denote the surface measure on $\VV_0^{d+1}$, which is well-defined for all polynomials
restricted on the conic surface. The orthogonal structure is studied in \cite{X20a} for a family of more general 
weight functions $t^\b (1-t)^\g$ with $\b > -d$ and $\g > -1$. We consider $\b = -1$ and $\g \ge -\f12$
because the addition formula for $\sw_{-1,\g}$ has a simpler form. 

Let $\CV_n(\VV_0^{d+1},\sw_{-1,\g})$ be the space of orthogonal polynomials of degree $n$. The reproducing
kernel $\sP_n(\sw_{-1,\g}; \cdot,\cdot)$ of this space satisfies an addition formula given below.

\begin{thm}  \label{thm:sfPbCone2}
Let $d \ge 2$ and $\g \ge -\f12$. Then, for $(x,t), (y,s) \in \VV_0^{d+1}$,
\begin{align} \label{eq:sfPbCone}
 \sP_n \big(\sw_{-1,\g}; (x,t), (y,s)\big) =  b_{\g,d}  \int_{[-1,1]^2} & Z_{n}^{(\g+d-\f32,-\f12)} \left(2 \zeta (x,t,y,s; v)^2-1 \right) \\
  & \times  (1-v_1^2)^{\f{d-4}{2}} (1-v_2^2)^{\g-\f12} \d v, \notag
\end{align} 
where $b_{\g,d}$ is a constant so that $\sP_0\big(\sw_{-1,\g}; (x,t), (y,s)\big) =1$ and 
\begin{equation} \label{eq:zeta}
 \zeta (x,t,y,s; v)  = v_1 \sqrt{\tfrac{st + \la x,y \ra}2}+ v_2 \sqrt{1-t}\sqrt{1-s};
\end{equation}
moreover, the identity holds under the limit \eqref{eq:limitInt} when $\g = -\f12$ and/or $d = 2$. 
\end{thm} 
 
The identity \eqref{eq:sfPbCone} is stated in \cite{X20a} in terms of $Z_{2n}^{\g+d-1}$, which we have rewritten using \eqref{eq:Z-quadr-trans} following \cite{X21}. It gives the addition formula \eqref{eq:7Pn} with $\a = \g+d-\f32$ and 
$\b = -\f12$. Following the Definition \ref{defn:7convol}, the convolution of $f \in L^1(\VV_0^{d+1}, \sw_{-1,\g})$ 
and $g \in L^1([-1,1],w_{\g+d-\f32,-\f12})$, denoted by $f *_\g g$ is defined by 
$$
 f *_\g g (x,t) := \bs_\g \int_{\VV_0^{d+1}} f(y,s) \sT^\g g\big((x,t),(y,s) \big) \sw_{-1,\g}(s) \d \s,
$$
where the operator $g \mapsto \sT^\g g$ is defined by  
$$
   \sT^\g g\big((x,t),(y,s) \big):=  b_{\g,d}  \int_{[-1,1]^2}  g \left(2 \zeta (x,t,y,s; v)^2-1 \right) 
        (1-v_1^2)^{\f{d-4}{2}} (1-v_2^2)^{\g-\f12} \d v.
$$
We can also define the operator $g \mapsto S_{\t,\sw} g$ as in \eqref{eq:Stheta} and use it to show that the 
convolution $f*_\g g$ satisfies Proposition \ref{prop:Stheta}, but this will not be needed below. The maximal
function $\CM_\sw$ defined in \eqref{eq:CMw}, denote by $\CM_\g$ for $\sw_{-1,\g}$, is given by
$$
   \CM_\g f(x,t) : = \sup_{0 \le \t \le \pi} \frac{f *_\g \chi_{[0,\t]} (x,t)}{ 2^\l \int_0^\t \big(\sin \f{\phi}{2}\big)^{2\l} \d \phi}
    \quad\hbox{with}\quad \l = \g+d-1,
$$
where we have used the elementary identity
$$
   \int_0^\t (1-\cos \phi)^{\l-\f12} (1+\cos\phi)^{-\f12} \sin \phi \,\d\phi = 2^\l  \int_0^\t \left(\sin \f{\phi}{2}\right)^{2\l} \d \phi.
$$
 
The Poisson integral $\sQ_r(\sw_{-1,\g};f)$, defined by \eqref{eq:Poisson}, is the convolution of $f$ with 
$q_r^{(\l-\f12,-\f12)}$, the latter is given by, as the limit of \eqref{eq:q_rPoisson},
$$
 q_r^{(\l-\f12,-\f12)}(u) = \f12 \Bigg[\frac{1-r}{ \big(1 - 2 \sqrt{r} \sqrt{\f{1+u}2} + r\big)^{\g+\d}}
       +\frac{1-r}{ \big(1 + 2 \sqrt{r} \sqrt{\f{1+u}2} + r\big)^{\g+\d}}\Bigg]. 
$$
We define the maximal Poisson integral operator on the conic domain by 
$$
  \sQ_*(\sw_{-1,\g}; f, (x,t)) = \sup_{0 < r<1}  \sQ_n(\sw_{-1,\g}; |f|, (x,t)). 
$$
The following theorem shows that the maximal function is equivalent to the maximal Poisson integral operator
on the conic surface. 

\begin{thm} \label{thm:sQr=CM}
Let $d \ge 2$ and $\g \ge -\f12$. If $f \in L^1(\VV_0^{d+1}; \sw_{-1,\g})$, then 
for every $(x,t) \in \VV_0^{d+1}$, 
$$
    c_1 \CM_\g f(x,t) \le \sQ_*(\sw_{-1,\g}; f, (x,t))  \le c_2 \CM_\g f(x,t).
$$
\end{thm} 

\begin{proof}
Let $\l = \g+d-1$. The right-hand side inequality is established in Theorem \ref{thm:PoissonM}. For 
the left-hand inequality, we observe that, for $1 - \sqrt{r} \sim \t$ and $0 \le \phi \le \t$, 
\begin{align*}
  q_r^{(\l-\f12,-\f12)}(\cos \phi) \, & \ge \f12 \frac{1-r}{ (1- 2 \sqrt{r} \cos \f \phi 2 +r )^{\l+1}}
      \ge    \frac{1-r}{ ((1-\sqrt{r})^2 + \sqrt{r} \phi^2)^{\l+1}} \\
      &  \ge c \frac{1-r}{ ((1-\sqrt{r})^2 + \sqrt{r} \t^2)^{\l+1}} 
    \ge c\, \t^{- 2\l -1},
\end{align*}
which implies, in particular, that  
$$
      \chi_{[\cos\t,1]}(\cos \phi) \le \t^{2\l+1}  q_r^{(\l-\f12,-\f12)}(\cos\phi), \qquad 0 \le \phi \le \pi. 
$$
Consequently, since $\int_0^\t (\sin \f{\phi}{2})^{2\l} \d \phi \sim \t^{2\l+1}$, for $0 \le \t \le \pi$ and $r \sim (1- c \t )^2$
we obtain  
$$
    \frac{ |f| \ast_\g  \chi_{[\cos\t,1]}(x,t) }{\t^{2\l+1}} \le c  |f| \ast_\g q_r^{(\l-\f12,-\f12)}(x,t) \le c\, \sQ_*(\sw_{-1,\g}; f, (x,t)).
$$
Taking supreme over $\t$ proves  $c_1 \CM_\g f(x,t) \le \sQ_*(\sw_{-1,\g}; f, (x,t))$.  
\end{proof}

\subsection{Boundedness of the Maximal function}
The Hardy-Littlewood maximal function for $(\VV_0^{d+1}, \sw_{-1,\g}, \sd_{\VV_0})$ is defined by 
$$
  M_\g f(x,t) := \sup_{0 \le \t \le \pi} \frac{1}{\sw_{-1,\g} \big(\sc((x,t), \t)\big)} \int_{\sc((x,t), \t)} |f(y,s)| \sw_{-1,\g}(s) \d \s (y,s), 
$$ 
where $\sc((x,t),\t)$ is the conic cap. The main result of the subsection shows that the maximal function
$\CM_\g f$ is bounded by $M_\g f$. More precisely, our main result in this subsection is the following theorem. 

\begin{thm}\label{thm:CM<M}
Let $d \ge 2$ and $\g\ge -\f12$. For $f \in L^1 \left(\VV_0^{d+1}, \sw_{-1,\g}\right)$, 
$$
   \CM_\g f(x,t) \le c \,M_\g f(x,t), \qquad (x,t) \in \VV_0^{d+1}. 
$$
\end{thm} 

\begin{proof}
Rewriting the integral over ${\sc((x,t), \t)}$ in $M_\g f$ as an integral over $\VV_0^{d+1}$,
$$
 \int_{\sc((x,t), \t)} |f(y,s)| \sw_{-1,\g}(s) \d \s (y,s) =  \int_{\VV_0^{d+1}}  |f(y,s)| \chi_{\sc((x,t), \t)}(y,x) 
     \sw_{-1,\g}(s) \d \s (y,s),
$$
where $\chi_{\sc((x,t), \t)}$ is the characteristic function of $\sc((x,t), \t)$, we see that it is sufficient to 
prove the inequality 
$$
  \left|\sT^\g \chi_{[\cos \t,1]}\big( (x,t), (y,s)\big) \right| \le c  \frac{1}{\sw_{-1,\g} \big(\sc((x,t), \f \t 2)\big)}
      \int_0^\t \left(\sin \f{\phi}{2}\right)^{2\l} \d \phi \chi_{\sc((x,t), \f{\t}2)}(y,s).
$$
Using $\int_0^\t (\sin \f{\phi}{2})^{2\l} \d \phi \sim \t^{2\l+1}$ and the relation \eqref{eq:sw-doubling}, the above
inequality is equivalent to the following inequality, for $(x,t), (y,s) \in \VV_0^{d+1}$, 
\begin{equation} \label{eq:Tchi-bd1}
  \left| \sT^\g \chi_{[\cos \t,1]}\big( (x,t), (y,s)\big) \right |\le c\frac{\t^{2\g+d-1}} {(t+\t^2)^{\f{d-2}{2}} (1-t+\t^2)^{\g +\f12} }
      \chi_{\sc((x,t), \f{\t}2)}(y,s).
\end{equation}
 
We prove this inequality for the case $d > 2$ and $\g > -\f12$. The degenerated cases of $d =2$ and/or 
$\g = -\f12$ can be handled similarly and are easier. By the definition of $\sT^\g$, our main task is to establish 
the inequality
\begin{align} \label{eq:Tchi-bd}
 \mathrm{LH}:= \int_{-1}^1  \int_{-1}^1  \chi_{[\cos \f\t 2,1]} & \left(|\zeta(x,t,y,s;v)| \right) 
        (1-v_1^2)^{\f{d-4}{2}} (1-v_2^2)^{\g-\f12} \d v \\
   &  \qquad\qquad \le c\frac{\t^{2\g +d-1}} {(t+\t^2)^{\f{d-2}{2}} (1-t+\t^2)^{\g +\f12} }  \notag
\end{align}
for $(y,s) \in \sc( (x,t), \tfrac \t 2)$.  
Indeed, since $2 \zeta(x,t,y,s; v)^2 -1 \le \cos \t$ is equivalent to $|\zeta(x,t,y,s;v)| \le \cos \frac{\t}2$, we obtain
$$
  \chi_{[\cos \t,1]} \left(2 \zeta(x,t,y,s;v)^2-1\right)  =   \chi_{[\cos \f\t 2,1]} \left( |\zeta(x,t,y,s;v)|\right). 
$$
Moreover, set $\mathbf{1} = (1,1)$, then $\cos \sd_{\VV_0} ((x,t), (y,s)) = \zeta(x,t,y,s; \mathbf{1}) \ge 0$, which 
implies 
$$
     \chi_{\sc ((x,t),\f{\t}{2})} = \chi_{[\cos \f{\t}2,1]} \left(\zeta(x,t,y,s; \mathbf{1})\right).
$$ 
Now, if $\zeta(x,t,y,s; \mathbf{1}) \le \cos \frac{\t}{2}$, then $\left| \zeta(x,t,y,s; v)\right| \le \cos \frac{\t}{2}$ 
for $-1 \le v_1,v_2 \le 1$. Hence, if $\chi_{\sc((x,t), \f{\t}2)}(y,s) = 0$ then $\chi_{[\cos \f\t 2,1]}(|\zeta(x,t,y,s;v|)) =0$ 
and, consequently, $\sT^\g  \chi_{[\cos \t,1]}\big( (x,t), (y,s)\big)=0$, which justifies the appearance of 
$\chi_{\sc((x,t),\f{\t}{2})}(y,s)$ in the right-hand side of \eqref{eq:Tchi-bd1}, so that it suffices to prove the inequality 
\eqref{eq:Tchi-bd}. Now, let $\eta (t,s; v) =  v_1 \sqrt{t s} + v_2 \sqrt{1-t}\sqrt{1-s}$. Then, using $|\la x, y \ra| \le t s$ for 
$(x,t), (y,s) \in \VV_0^{d+1}$, we see that $| \zeta(x, t, y, s; v)|$ is bounded by both $\eta(t, s; (|v_1|,1))$ and 
$\eta(t, s; (1,|v_2|))$. The argument that justifies the appearance of $\chi_{\sc((x,t),\f{\t}{2})}(y,s)$ in \eqref{eq:Tchi-bd1}
also shows
\begin{align*}
& \chi_{[\cos \f{\t}2,1]}\left( |\zeta(x,t,y,s; v)| \right)  \le  \chi_{[\cos \f \t 2,1]}\left (\eta(t, s; (|v_1|,1)) \right)  
 \chi_{[\cos \f \t 2,1]}\left (\eta(t, s; (1,|v_2|)) \right),
\end{align*}
which implies immediately that 
\begin{align*}
 \mathrm{LH} \,& \le \int_{-1}^1 \chi_{[\cos \f\t 2,1]}\left(|v_1| \sqrt{t s} + \sqrt{1-t}\sqrt{1-s}\right) (1-v_1^2)^{\f{d-4}{2}} \d v_1 \\
     & \,\,\,\,\, \times  \int_{-1}^1 \chi_{[\cos \f \t 2,1]}\left(\sqrt{t s} + |v_2|\sqrt{1-t}\sqrt{1-s}\right) (1-v_2^2)^{\g-\f12} \d v_2 \\
     & = I_1 \times I_2. 
\end{align*}
For $t = \cos^2 \f \t 2$ and $s = \cos^2 \f\phi 2$ in $[0,1]$, the function 
$$
\sd_{[0,1]}(t,s) := \arccos \left(t s + \sqrt{1-t}\sqrt{1-s} \right) = \tfrac12 |\t - \phi|
$$ 
defines a distance on the interval $[0,1]$. From \cite[(4.2)]{X21}, we deduce 
$$
   \sd_{[0,1]}(t,s) \le \sd_{\VV_0}((x,t), (y,s)) \quad \hbox{for} \quad (x,t), (y,s) \in \VV_0^{d+1}.
$$
Hence, if $(y,s) \in \sc( (x,t), \tfrac \t 2)$, then $\sd_{[0,1]}(t,s) \le  \tfrac \t 2$. Thus, the proof of \eqref{eq:Tchi-bd}
is reduced to establish the inequalities 
$$
   I_1  \le c \frac{\t^{d-2}} {(t+\t^2)^{\f{d-2}{2}} } \quad \hbox{and} \quad
   I _2 \le c \frac{\t^{2\g+1}} {(1-t+\t^2)^{\g +\f12} } 
$$
for $t, s$ satisfying $\sd_{[0,1]}(t,s) \le \tfrac \t 2$.

We consider $I_1$ first. If $0 \le \sqrt{t} \le 2 \t$, then $t + \t^2 \sim \t^2$ and the desired bound for $I_1$ holds trivially. 
Hence, we only need to consider the case $\sqrt{t} \ge 2 \t$. Using symmetry and changing variable 
$v_1 \mapsto u = v_1 \sqrt{t s} +  \sqrt{1-t}\sqrt{1-s}$, the integral $I_1$ is bounded by 
\begin{align*}
    & I_1\le  2^{\g + \f12}\int_0^1 \chi_{[\cos \f \t 2,1]}\left( v \sqrt{t s} + \sqrt{1-t}\sqrt{1-s}\right) (1-v)^{\f{d-4}{2}} \d v \\   
    & =  2^{\g + \f12} \int_{\sqrt{1-t}\sqrt{1-s}}^{\sqrt{t s} + \sqrt{1-t}\sqrt{1-s}} \chi_{[\cos \f\t 2, 1]}(u)  
          \left (1- \frac{u- \sqrt{1-t}\sqrt{1-s} }{\sqrt{ ts}} \right)^{\f{d-4}2} \frac{\d u}{\sqrt{t s }} \\
    & \le c \frac{1}{ (\sqrt{t s})^{\f{d-2}{2}} } \int_{\cos \f \t 2}^{\cos \sd_{[0,1]} (t,s)}
         \big( \cos \sd_{[0,1]} (t,s) -  u\big)^{\f{d-4}{2}} \d u \\
    & \le c  \frac{1}{ (\sqrt{t s})^{\f{d-2}{2}} }   \big( \cos \sd_{[0,1]} ( t,s) -  \cos \t \big)^{\f{d-2}{2}} \\
    &  \le c  \frac{1}{ (\sqrt{t s})^{\f{d-2}{2}} } |\sd_{[0,1]} (t,s)-\t| \cdot  |\sd_{[0,1]} (t,s)+\t|  \\
    &  \le c  \frac{\t^{d-2}}{ (\sqrt{t s})^{\f{d-2}{2}} } 
\end{align*} 
for $\sd_{[0,1]} (t,s) \le \f \t 2$. Since $\sqrt{t} \ge 2 \t$, it follows $\sqrt{t} \ge (\sqrt{t} + 2 \t)/2$. Moreover, 
using $t = \cos^2 \f \t 2$ and $s = \cos^2 \f \phi 2$, it is easy to see that $|\sqrt{t} - \sqrt{s} | \le \sd_{[0,1]}(t,s)$, 
so that $\sqrt{s} \ge \sqrt{t} - \f \t 2 \ge 
(\sqrt{t} + \t)/2$ if $(y,s) \in \sc( (x,t), \frac \t 2)$. Thus, $\sqrt{ts} \ge \f14 (t + \t^2)$,
from which the desired upper bound for $I_1$ follows. The proof for $I_2$ follows similarly, using
$|\sqrt{1-t} - \sqrt{1-s} | \le \sd_{[0,1]}(t,s)$. This completes the proof. 
\end{proof}

Since $\CM_\g f$ is dominated by the Hardy-Littlewood maximal function, it follows immediately that it
possesses the usual properties of maximal functions. 

\begin{cor}
Let $\d \ge 2$ and $\g \ge -\f12$. For $f \in L^p(\VV_0^d, \sw_{-1,\g})$, $1 < p < \infty$, 
$$
  \|\CM_\g f\|_{p,\sw_{-1,\g}} \le c_p \| f\|_{p,\sw_{-1,\g}}, \qquad 1 < p < \infty,
$$ 
and, for $f \in L^1(\VV_0^d, \sw_{-1,\g})$ and $\a > 0$, 
$$
   \sw_{-1,\g} \left\{(x,t) \in \VV_0^{d+1}: \CM_\g f(x,t) > \a \right\} \le c \frac{\|f\|_{p, \sw_{-1,\g}}}{\alpha}.
$$
\end{cor}
 
\subsection{Multiplier theorem}
The maximal $(C,\delta)$ operator with respect to $\sw_{-1,\g}$ is denoted by, in terms of the $(C,\delta)$ 
operators with respect to $\sw_{-1,\g}$, 
$$
    S_*^\delta (\sw_{-1,\g}; f, (x,t)) = \sup_{n \ge 0}  \left|S_n^\delta (\sw_{-1,\g}; f, (x,t))\right|, \quad (x,t) \in \VV_0^{d+1}.
$$
Because of Theorem \ref{thm:CM<M}, the Assertion 1 holds for $S_*^\delta(\sw_{-1,\g}; f)$ if $\delta \ge \l +1 = d+\g$
by \eqref{eq:S*<CM}. The restriction on $\delta$, however, can be relaxed for the conic surface. 

\begin{thm} \label{thm:AssertionV0}
Let $d\ge 2$ and $\g \ge -\f12$. If $\delta > \g+d-1$ and $f\in L^1(\VV_0^{d+1}, \sw_{-1,\g})$, then for 
every $(x,t) \in \VV_0^{d+1}$, 
$$
    S_*^\delta (\sw_{-1,\g}; f, (x,t))  \le c\, \CM_\g f(x,t). 
$$
\end{thm} 

\begin{proof}
Let $\l = \g+d-1$. Since the case $\delta \ge \l+1$ is already established, we only need to consider $\l < \delta < \l +1$. 
By the estimate of $k_n^\delta(w_{\a,\b}; u,1)$ in \cite[(9.4.4) and (9.41.14)]{Sz}, 
\begin{align*}
  \left| k_n^\delta(w_{\l-\f12,-\f12}; \cos \t,1) \right| \,& \le c n^{\l - \delta} \left[ (1-\cos\t +n^{-2})^{- (\delta + \l+\f12)/2} +1 \right] \\
       & \le c n^{\l - \delta} \left[ (n^{-1} + \t)^{- (\delta + \l+\f12)} +1 \right] =: k_n(\t).
\end{align*}
It is easy to see that $k(\t)$ satisfy the conditions of Theorem \ref{thm:f*g<M}. Moreover, 
$$
   \int_0^\pi k_n(\phi) \left (\sin \frac \phi 2\right )^{2\l} \d \phi \le 
       c n^{\l-\delta} \int_0^\pi  \left[ (n^{-1} + \phi)^{- (\delta + \l+\f12)} +1 \right] \phi^{2\l} \d \t
$$ 
is bounded by a constant depending only on $\g$ and $d$, as can be easily seen by considering 
$0 \le \phi \le n^{-1}$ and $\phi \ge n^{-1}$ separately. Consequently, we obtain 
$$
  \left|S_n^\delta \left(\sw_{-1,\g}; f, (x,t)\right)\right|  \le c \CM_\g f(x,t)
$$
for all $n \ge 0$. Taking supreme over $n$ completes the proof. 
\end{proof}

This shows, by Theorem \ref{thm:CM<M}, that the Assertion 1 holds for $\delta > \l$. In particular, by
Theorem \ref{thm:multiplier}, the following multiplier theorem holds on the conic surface.

\begin{thm}\label{thm:multiplierV0} 
Let $d\ge 2$ and $\g \ge -\f12$.  Let $\{\mu_j\}_{j=0}^\infty$ be a bounded sequence of real numbers such that 
$$
  \sup_{j \ge 0} 2^{j(k-1)} \sum_{\ell =2^j +1}^{2^{j+1}} \left|\triangle^k \mu_{\ell}\right| \le M < \infty
$$
for $k \ge \lfloor d+\g\rfloor$. Then $\{\mu_j\}$ defines an $L^p(\VV_0, \sw_{-1,\g})$ multiplier for $1 < p < \infty$; that is
$$
   \left \| \sum_{k =0}^\infty \mu_k \proj_k (\sw_{-1,\g}; f)\right \|_{p,\sw_{-1,\g}} \le c  \|f\|_{p,\sw_{-1,\g}}, \quad 1 < p < \infty, 
$$
where $c$ is independent of $\mu_j$ and $f$. 
\end{thm}

\section{Fourier orthogonal series on the solid conic}
\setcounter{equation}{0}

In this section we work with the cone $\RR^{d+1}$ defined, for $d \ge 1$, by 
$$
     \VV^{d+1}= \{(x,t): \|x\| \le t, \, x \in \RR^d, \, 0 \le t \le 1\},
$$
which is the domain bounded by the conic surface $\VV_0^{d+1}$ and the hyperplane $t =1$. The development 
on this domain is similar to that on the conic surface but has another layer of complication. The structure of this 
section is parallel to that of the previous one. 

\subsection{Analysis on the conic}
For $\mu \ge 0$ and $\g \ge -\f12$, let $W_{\g,\mu}$ be the weight function defined by 
$$
   W_{\g,\mu}(x,t) = (1-\|x\|^2)^{\mu-\f12} (1-t)^\g,  \qquad (x,t) \in \VV^{d+1}. 
$$
We consider the homogeneous space $(\VV^{d+1}, W_{\g,\mu}, \sd_{\VV})$, where 
$\sd_{\VV}$ is the distance function on $\VV^{d+1}$, defined by
\begin{equation*}
  \sd_{\VV} ((x,t), (y,s)) : =  \arccos \left(\sqrt{\frac{\la x,y\ra+\sqrt{t^2-\|x\|^2}\sqrt{s^2-\|y\|^2} +t s }2}  + \sqrt{1-t}\sqrt{1-s}\right)
\end{equation*}
for $(x,t), (y,s) \in \VV^{d+1}$. For $(x,t) \in \VV^{d+1}$ and $r > 0$, the ball centered at $(x,t)$ with radius $r$ is 
denoted by 
$$
      \cb((x,t), r): = \left\{ (y,s) \in \VV^{d+1}: \sd_{\VV} \big((x,t),(y,s)\big)\le r \right\}.
$$   
Then $W_{\mu,\g}$ is a doubling weight with respect to $\sd_{\VV}$ since it satisfies \cite[(5.5)]{X21}
\begin{align} \label{eq:W-doubling}
 W_{\g,\mu} \big(\cb((x,t), r)\big):= \,& \bb_{\g,\mu} \int_{ \cb((x,t), r)} W_{\g,\mu} (y,s) \d y \d s  \\
          \sim \, & r^{d+1} (t+ r^2)^{\f{d-1}{2}} (1-t+ r^2)^{\g+\f12} (t^2-\|x\|^2+ r^2)^{\mu}, \notag
\end{align}
where $\bb_{\mu,\g}$ is the normalized constant defined by $\bb_{\mu,\g} \int_{\VV^{d+1}} W_{\g,\mu} (x,t) \d x \d t =1$.

The orthogonality on $L^2(\VV^{d+1},W_{\g,\mu})$ is defined in terms of the inner product 
$$
\la f, g\ra_{W} =\bb_{\mu,\g} \int_{\VV^{d+1}} f(x,t) g(x,t) W_{\g,\mu}(x,t)\d x \d t.
$$ 
The orthogonal structure is studied in \cite{X20a} for a family of more general weight functions
$t^\b (1-t)^\g (t^2-\|x\|^2)^{\mu-\f12}$ with $\b > -d$ and $\g > -1$ and $\mu > -\f12$. We consider $\b = 0$, 
$\g \ge -\f12$ and $\mu \ge 0$ because the addition formula for $W_{\g,\mu}$ is more manageable.

Let $\CV_n(\VV^{d+1},W_{\g,\mu})$ be the space of orthogonal polynomials of degree $n$. The reproducing
kernel $\Pb_n(W_{\g,\mu}; \cdot,\cdot)$ of this space satisfies a closed form formula given below.

\begin{thm} \label{thm:PnCone2}
Let $d \ge 1$, $\mu \ge 0$ and $\g \ge -\f12$. Then 
\begin{align}\label{eq:PbCone2}
  \Pb_n \big(W_{\g,\mu}; (x,t), (y,s)\big) =\, & 
   c_{\mu,\g,d} \int_{[-1,1]^3}  Z_n^{(\g+d+2 \mu-\f12,-\f12)} \left(2 \xi (x, t, y, s; u, v)^2-1\right) \\
     &\times   (1-u^2)^{\mu-1} (1-v_1^2)^{\mu + \f{d-3}2}(1-v_2^2)^{\g-\f12}  \d u \d v, \notag
\end{align}
where $c_{\mu,\g,d}$ is a constant, so that $\Pb_0 =1$ and 
$\xi (x,t, y,s; u, v) \in [-1,1]$ is defined by 
\begin{align} \label{eq:xi}
\xi (x,t, y,s; u, v) = &\, v_1 \sqrt{\tfrac12 \left(ts+\la x,y \ra + \sqrt{t^2-\|x\|^2} \sqrt{s^2-\|y\|^2} \, u \right)}\\
      & + v_2 \sqrt{1-t}\sqrt{1-s}. \notag
\end{align}
When $\mu = 0$ and $d =1$ or $\mu =0$ and/or $\g = -\f12$, the identity \eqref{eq:PbCone2} holds under the limit \eqref{eq:limitInt}.
\end{thm}

The identity \eqref{eq:PbCone2} is stated in \cite{X20a} in terms of $Z_{2n}^{\l}$, which we have rewritten 
using \eqref{eq:Z-quadr-trans}. This gives the addition formula \eqref{eq:7Pn} with $\a = \g+d+2 \mu-\f12$ and 
$\b = -\f12$. Following Definition \ref{defn:7convol}, the convolution of $f \in L^1(\VV^{d+1}, W_{\g,\mu})$ 
and $g \in L^1([-1,1],w_{\g+d+2 \mu-\f12,-\f12})$, denoted by $f *_{\g,\mu} g$, is defined by 
$$
 f *_{\g,\mu} g (x,t) := \bb_{\mu,\g} \int_{\VV^{d+1}} f(y,s) \Tb^{\g,\mu} g \big((x,t),(y,s) \big) W_{\g,\mu}(y,s) \d y \d s,
$$
where the operator $g \mapsto \Tb^{\g,\mu} g$ is defined by  
\begin{align*}
   \Tb^{\g,\mu} g\big((x,t),(y,s) \big):= \,& c_{\mu,\g,d}  \int_{[-1,1]^3}  g \left(2 \xi(x,t,y,s; u,v)^2-1 \right)  \\
         &\times   (1-u^2)^{\mu-1} (1-v_1^2)^{\a - 1}(1-v_2^2)^{\g-\f12}  \d u \d v. \notag
\end{align*}
The maximal function $\CM_\sw$ defined in \eqref{eq:CMw}, denote by $\bCM_{\mu,\g}$ for $W_{\g,\mu}$, is given by
$$
   \bCM_{\g,\mu} f(x,t) : = \sup_{0 \le \t \le \pi} \frac{f *_{\g,\mu} \chi_{[0,\t]} (x,t)}
         { 2^\l \int_0^\t \big(\sin \f{\phi}{2}\big)^{2\l} \d \phi}
    \quad\hbox{with}\quad \l = \g + d + 2 \mu.
$$

Let $ \Qb_n(W_{\g,\mu}; f)$ denote the Poisson integral operator on $\VV^{d+1}$, defined as in \eqref{eq:Poisson} 
with $\sw = W_{\g,\mu}$. As in  the case of conic surface, the above maximal function is equivalent to the maximal 
Poisson integral operator
$$
  \Qb_*(W_{\g,\mu}; f, (x,t)) = \sup_{0 < r<1}  \Qb_n(W_{\g,\mu}; |f|, (x,t)). 
$$
Indeed, the Poisson operator is again the convolution of $f$ with $q_r^{(\l-\f12,-\f12)}$, where $\l = \g + d + 2 \mu$,
and the proof of Theorem \ref{thm:sQr=CM} requires only an estimate on $q_r^{(\l-\f12,-\f12)}$; hence, the proof of 
the theorem below follows exactly as in the case of the conic surface. 

\begin{thm}
Let $d \ge 1$, $\g \ge -\f12$ and $\mu \ge 0$. If $f \in L^1(\VV^{d+1}; W_{\g,\mu})$, then 
for every $(x,t) \in \VV_0^{d+1}$, 
$$
    c_1 \bCM_{\g,\mu} f(x,t) \le \Qb_*(W_{\g,\mu}; f, (x,t))  \le c_2 \bCM_{\g,\mu} f(x,t).
$$
\end{thm} 

\subsection{Boundedness of the Maximal function}
The Hardy-Littlewood maximal function for $(\VV^{d+1}, W_{\g,\mu}, \sd_{\VV})$ is defined by 
$$
  \Mb_{\g,\mu} f(x,t) := \sup_{0 \le \t \le \pi} \frac{1}{W_{\g,\mu} \big(\cb((x,t), \t)\big)} \int_{\cb((x,t), \t)} |f(y,s)| 
     W_{\g,\mu}(y,s) \d y \d s, 
$$ 
where $\cb((x,t),\t)$ is the ball on the cone. We show that the maximal function $\bCM_{\g,\mu} f$ is bounded by 
$\Mb_{\g,\mu} f$, just like in the case of conic surface. 

\begin{thm}\label{thm:CM<M-V}
Let $d \ge 1$, $\g\ge -\f12$ and $\mu \ge 0$. For $f \in L^1 \left(\VV^{d+1}, W_{\g,\mu}\right)$, 
$$
   \bCM_{\g,\mu} f(x,t) \le c \, \Mb_{\g,\mu} f(x,t), \qquad (x,t) \in \VV^{d+1}. 
$$
\end{thm} 

\begin{proof}
The proof is parallel to that of Theorem \ref{thm:CM<M}. Let $\l = \g + d + 2 \mu$. Following the argument
of the conic surface, the proof is reduced to show 
\begin{equation} \label{eq:Tchi-bdV}
  \left| \Tb^{\g,\mu} \chi_{[\cos \t,1]}\big( (x,t), (y,s)\big) \right |\le 
    c \frac{\t^{2\l - d }  \chi_{\cb((x,t), \f{\t}2)}(y,s)} {(t+\t^2)^{\f{d-1}{2}} (1-t+\t^2)^{\g +\f12} (t^2-\|x\|^2+\t^2)^\mu}.
\end{equation}
We again prove only the non-degenerated case with $d > 1$, $\g > -\f12$ and $\mu > 0$. By the definition of 
$\Tb^{\g,\mu}$ and $\chi_{\cb ((x,t),\f{\t}{2})}(y,s) = \chi_{[\cos \f{\t}2,1]} \left(\xi(x,t,y,s; 1,\mathbf{1})\right)$, using
the argument in the proof of Theorem \ref{thm:CM<M}, we see that the main task for
proving \eqref{eq:Tchi-bdV} is to establish the inequality
\begin{align*}
 \mathrm{LH}:= \int_{[-1,1]^3} \chi_{[\cos \f\t 2,1]} & \left(|\xi(x,t,y,s;u,v)| \right) 
       (1-u^2)^{\mu-1} (1-v_1^2)^{\mu+\f{d-3}{2}} (1-v_2^2)^{\g-\f12} \d v \\
   & \qquad\qquad \le  c\frac{\t^{2\l-d}}  {(t+\t^2)^{\f{d-1}{2}} (1-t+\t^2)^{\g +\f12} (t^2-\|x\|^2+\t^2)^\mu}
\end{align*}
for $(y,s) \in \cb( (x,t), \tfrac \t 2)$. Using $|\la x,y \ra + \sqrt{t^2-\|x\|^2} \sqrt{s^2-\|y\|^2} |\le t s$, it is easy 
to see that 
$|\xi(x,t,y,s;u,v)|$ is bounded by both $\eta(t,s; (|v_1|,1))$ and $\eta(t,s; (1,|v_2|))$, where 
$\eta(t,s; v) = v_1 \sqrt{t s} + v_2 \sqrt{1-t}\sqrt{1-s}$. Moroever, by $|v_1| \le 1$ and $|v_2|\le 1$, 
$|\xi(x,t,y,s;u,v)|$ is also bounded by $\xi(x,t,y,s;u,\mathbf{1})$. Consequently, we obtain 
\begin{align*}
 \mathrm{LH} \,& \le c \int_{-1}^1 \chi_{[\cos \f\t 2,1]}\left(|v_1| \sqrt{t s} + \sqrt{1-t}\sqrt{1-s}\right)
       (1-v_1^2)^{\mu+\f{d-3}{2}} \d v_1 \\
       & \,\, \, \, \,  \times  \int_{-1}^1 \chi_{[\cos \f \t 2,1]}\left(\sqrt{t s} + |v_2|\sqrt{1-t}\sqrt{1-s}\right) (1-v_2^2)^{\g-\f12} \d v_2 \\
       & \,\, \, \,\,  \times  \int_{-1}^1 \chi_{[\cos \f \t 2,1]}\left(\xi(x,t,y,s;u,\mathbf{1})\right) (1-u^2)^{\mu-1} \d u \\
       & = I_1 \times I_2 \times I_3. 
\end{align*}
The first two terms, $I_1$ and $I_2$, can be estimated exactly as in Theorem \ref{thm:CM<M}, using
$$
  \big| \sqrt{t} - \sqrt{s} \big|\le \sd_{\VV} ((x,t), (y,s)) \quad \hbox{and} \quad 
      \big| \sqrt{1-t} - \sqrt{1-s} \big| \le \sd_{\VV} ((x,t), (y,s))
$$
given in \cite[Lemma 5.4]{X21}. Thus, it remains to prove 
$$
   I_3  \le c \frac{\t^{2\mu}} {\left(t^2-\|x\|^2+\t^2\right)^{\mu} }, \qquad (y,s) \in \cb( (x,t), \tfrac \t 2). 
$$
This inequality holds trivially if $t^2-\|x\|^2 \le 4 \t^2$. We assume $t^2 -\|x\|^2 \ge 4 \t^2$ below, so that
$t^2 -\|x\|^2 \ge (t^2 -\|x\|^2 + \t)/2$. Moroever, by \cite[Lemma 5.4]{X21}, 
$$
|\sqrt{t^2-\|x\|^2} - \sqrt{s^2-\|y\|^2}| \le \sd_\VV((x,t),(y,s)),
$$ 
it follows that, for $(y,s) \in \cb( (x,t), \tfrac \t 2)$, 
$$
s^2 - \|y\|^2 \ge \left(\sqrt{t^2-\|x\|^2} - \t /2\right)^2 \ge (t^2 -\|x\|^2 + \t)/4. 
$$
Let $\phi(t,s) = \sqrt{1-t}\sqrt{1-s}$. Since $\xi(x,t,y,s;1,\mathbf{1}) = \arccos \sd_\VV((x,t),(y,s))$, we can
make a change of variables $u \mapsto  v= \xi(x,t,y,s;u,\mathbf{1})$ in the integral of $I_3$ to obtain 
\begin{align*}
    I_3 \, & \le \int_{-1}^1  \chi_{[\cos \f \t 2,1]}\left(\xi(x,t,y,s;u,\mathbf{1})\right) (1-u)^{\mu-1} \d u \\
    & =  \int_{\xi(x,t,y,s;-1,\mathbf{1})}^{\xi(x,t,y,s;1,\mathbf{1})} \chi_{[\cos \f\t 2, 1]}(v)  
          \left (1- \frac{2 (v-\phi(t,s))^2 - (t s + \la x, y\ra)}{\sqrt{t^2-\|x\|^2}\sqrt{s^2-\|y\|^2}} \right)^{\mu-1} \\
       &   \qquad \qquad \qquad \qquad \qquad \qquad 
             \times  \frac{4(v-\phi(t,s))}{\sqrt{t^2-\|x\|^2}\sqrt{s^2-\|y\|^2}} \d v \\           
       & \le c \int_{\cos \f \t 2}^{\cos \sd_{\VV} ((x,t), (y,s))}
         \left( \big(\cos \sd_{\VV} ( (x,t), (y,s)) - \phi(t,s)\big)^2 - (v-\phi(t,s))^2 \right)^{\mu-1}   \\
           &   \qquad \qquad \qquad \qquad \qquad \qquad  \times  \frac{v-\phi(t,s)}{\left(t^2-\|x\|^2+\t^2\right)^{\mu} }\d v. 
\end{align*}
The last integral can be integrated out, so that 
\begin{align*}  
    I_3  & \le  c  \frac{1}{\left(t^2-\|x\|^2+\t^2\right)^{\mu} }
       \left( \big(\cos \sd_{\VV} ( (x,t), (y,s)) - \phi(t,s)\big)^2 - \big(\cos \tfrac \t 2 -\phi(t,s)\big)^2 \right)^\mu \\
    &  \le c  \frac{1}{\left(t^2-\|x\|^2+\t^2\right)^{\mu} }\left(\cos \sd_{\VV} ( (x,t), (y,s)) - \cos \tfrac \t 2 \right)^\mu  \\
    &  \le c  \frac{\t^{2\mu}}{ \left(t^2-\|x\|^2+\t^2\right)^{\mu}}.
\end{align*} 
This establishes the desired estimate for $I_3$ and completes the proof.  
\end{proof}

As in the case of conic surface, the $\bCM_{\g,\mu}$ inherits the properties of the Hardy-Littlewood 
maximal function $\Mb_{\g,\mu}$. 

\begin{cor}
Let $\d \ge 1$, $\g \ge -\f12$ and $\mu \ge 0$. For $f \in L^p(\VV^d, W_{\g,\mu})$, $1 < p < \infty$, 
$$
  \|\bCM_{\g,\mu} f\|_{p,W_{\g,\mu}} \le c_p \| f\|_{p,W_{\g,\mu}}, \qquad 1 < p < \infty,
$$ 
and, for $f \in L^1(\VV^d, W_{\g,\mu})$ and $\a > 0$, 
$$
  W_{\g,\mu} \left\{(x,t) \in \VV^{d+1}: \bCM_{\g,\mu} f(x,t) > \a \right\} \le c \frac{\|f\|_{p,W_{\g,\mu}}}{\alpha}.
$$
\end{cor}
 
\subsection{Multiplier theorem}
The Ces\`aro $(C,\delta)$ operator with respect to $W_{\g,\mu}$ and its maximal $(C,\delta)$ operator 
are denoted by  
$$
    S_*^\delta (W_{\g,\mu}; f, (x,t)) = \sup_{n \ge 0}  \left|S_n^\delta (W_{\g,\mu}; f, (x,t))\right|, \quad (x,t) \in \VV^{d+1}.
$$
We show that the Assertion 1 holds with $\delta_0 = \g+d+2\mu$, which is sharper than the one follows from
\eqref{eq:S*<CM} and Theorem \ref{thm:CM<M-V}. 

\begin{thm}
Let $d\ge 1$, $\g \ge -\f12$ and $\mu \ge 0$. If $\delta > \g+d + 2\mu$ and $f\in L^1(\VV^{d+1}, W_{\g,\mu})$, then 
for every $(x,t) \in \VV^{d+1}$, 
$$
    S_*^\delta (W_{\g,\mu}; f, (x,t))  \le c\, \bCM_{\g,\mu} f(x,t). 
$$
\end{thm} 

The proof of this theorem follows exactly as in that of Theorem \ref{thm:AssertionV0}. Consequently, by
Theorem \ref{thm:multiplier}, we obtain the following multiplier theorem on the cone. 

\begin{thm}\label{thm:multiplierV} 
Let $d\ge 1$, $\g \ge -\f12$ and $\mu \ge 0$.  Let $\{\mu_j\}_{j=0}^\infty$ be a bounded sequence of real 
numbers such that 
$$
  \sup_{j \ge 0} 2^{j(k-1)} \sum_{\ell =2^j +1}^{2^{j+1}} \left|\triangle^k \mu_{\ell}\right| \le M < \infty
$$
for $k \ge \lfloor d+\g + 2\mu \rfloor$. Then $\{\mu_j\}$ defines an $L^p(\VV, W_{\g,\mu})$ multiplier for 
$1 < p < \infty$; that is
$$
   \left \| \sum_{k =0}^\infty \mu_k \proj_k (W_{\g,\mu}; f)\right \|_{p,W_{\g,\mu}} \le c  \|f\|_{p,W_{\g,\mu}}, 
         \quad 1 < p < \infty, 
$$
where $c$ is independent of $\mu_j$ and $f$. 
\end{thm}

\end{document}